\documentclass[12pt]{iopart}

\usepackage{iopams}
\usepackage{amssymb,amsthm}		
\usepackage{graphicx}
\usepackage{bm} 
\usepackage[ruled,linesnumbered,vlined]{algorithm2e} 
\usepackage[caption=true]{subfig} 
\usepackage[cal=cm]{mathalfa} 

\usepackage[colorlinks=true]{hyperref}
\hypersetup{urlcolor=blue, citecolor=red}

\expandafter\let\csname equation*\endcsname\relax
\expandafter\let\csname endequation*\endcsname\relax
\usepackage{amsmath}

\newtheorem{theorem}{Theorem}
\newtheorem{corollary}[theorem]{Corollary}
\newtheorem{remark}[theorem]{Remark}
\newtheorem{definition}[theorem]{Definition}
\newtheorem{example}[theorem]{Example}

\newcommand{\Rdown}{R_{\downarrow}}
\newcommand{\Rup}{R_{\uparrow}}

\newcommand{\mathbcal}[1]{\mathrm{#1}}

\begin{document}
\title[fg-ORKA for object tracking]{A fast and gridless ORKA algorithm for tracking moving and deforming objects} 

\author{Florian Bossmann$^{1,4}$, Jianwei Ma$^{1,2}$ and Wenze Wu$^3$}

\address{$^1$ Harbin Institute of Technology, School of Mathematics, Harbin, China.}
\address{$^2$ Peking University, School of Earth and Space Science, Beijing, China.}
\address{$^3$ Xuzhou Heavy Machinery Co., Ltd, Xuzhou, China.}
\address{$^4$ Supported by NSFC research grant 42004109.}

\begin{abstract} 
Identifying objects in given data is a task frequently encountered in many applications. Finding vehicles or persons in video data, tracking seismic waves in geophysical exploration data, or predicting a storm front movement from meteorological measurements are only some of the possible applications. In many cases, the object of interest changes its form or position from one measurement to another. For example, vehicles in a video may change its position or angle to the camera in each frame. Seismic waves can change its arrival time, frequency, or intensity depending on the sensor position. Storm fronts can change its form and position over time. This complicates the identification and tracking as the algorithm needs to deal with the changing object over the given measurements.

In a previous work, the authors presented a new algorithm to solve this problem - Object reconstruction using K-approximation (ORKA). The algorithm can solve the problem at hand but suffers from two disadvantages. On the one hand, the reconstructed object movement is bound to a grid that depends on the data resolution. On the other hand, the complexity of the algorithm increases exponentially with the resolution. We overcome both disadvantages by introducing an iterative strategy that uses a resampling method to create multiple resolutions of the data. In each iteration the resolution is increased to reconstruct more details of the object of interest. This way, we can even go beyond the original resolution by artificially upsampling the data. We give error bounds and a complexity analysis of the new method. Furthermore, we analyze its performance in several numerical experiments as well as on real data. We also give a brief introduction on the original ORKA algorithm. Knowledge of the previous work is thus not required.
\end{abstract}

\noindent{\it Keywords}: object reconstruction, multiple measurements, column shift operator, multiresolution, data resampling.

\maketitle

\section{Introduction}

In many applications, one has to extract the desired information out of the given data. Often, the crucial step in this process can essentially be described as follows: Find and track the "object" withing the data that carries the relevant information. For example, to avoid collisions and accidents, an autonomous car first needs to identify objects close to it and track their movement \cite{ChenX17}. In geophysical exploration one seeks subsurface resources using e.g., seismic measurements. Such resource reservoirs are usually identified by tracking the seismic waves reflected from the reservoir boundary \cite{Shah73}. For accurate forecasts and severe weather warnings it is necessary to track and predict the movement of storms \cite{Niemczynowicz87}. There are many more applications that face the same problem, such as medical imaging \cite{Olesen11}, industrial processing \cite{Herwig13}, and many more \cite{Yilmaz06}.

From a mathematical perspective, the problem reads as follows. Given some measurement data $d\in\mathbb{R}^{M_1}$, find $L$ objects that best fit the measurement,i.e., solve
\begin{align*}
	Ax=d && \text{where }x=\sum\limits_{k=1}^L\text{Object}_k\in\mathbb{R}^{M_2}.
\end{align*}
Here $A\in\mathbb{R}^{M_1\times M_2}$ is the measurement matrix, i.e., the linear operator that describes the measuring process. A commonly used assumption is, that the number of (relevant) objects within the data is much smaller than the data itself. In other words $L\ll M_1$ and we say that the data is sparse (in some representation). To reconstruct $x$ in the simplest case we can solve
\begin{align}\label{vector0}
	\min\limits_{x\in\mathbb{R}^{M_2}}\|Ax-d\|_2^2+\mu\|x\|_0.
\end{align} 
Here, we already assume that the data is noised and minimize over a data fidelity term instead of forcing the exact equality $Ax=d$. The $0$-norm is defined as $\|x\|_0=\#\{k\ |\ x_k\neq0\}$ and actually not a norm but abuse of notation. The parameter $\mu$ weights the sparsity promoting term against the data fidelity term. Problem (\ref{vector0}) is NP-hard and thus there are no efficient algorithms to solve it exactly \cite{Foucart13}. The most common used approaches involve replacing the $0$-norm with a $1$-norm \cite{Rudelson06} or using Greedy methods \cite{Tropp04}. In (\ref{vector0}) it is assumed that the vector $x$ itself is sparse. A more practical approach is, to assume that $x$ is sparse in some presentation, i.e., $x=\Phi y$ for some matrix $\Phi$ and $\|y\|_0$ is small. By simply replacing $A$ with $A\Phi$ in (\ref{vector0}) we can use the same algorithms to solve for $y$ and then reconstruct $x$ from $y$. The columns of the matrix $\Phi$ can be interpreted as the "objects" that we are looking for. Then a sparse $y$ with $\|y\|_0=L$ means that we found exactly $L$ objects in our data. There are mathematically motivated choices for $\Phi$ that produce sparse representations for many different kinds of data. For example, we can use a Wavelet basis \cite{Mallat99}, Fourier basis \cite{Plonka18}, or trigonometric functions \cite{Ahmed74}. We can also combine several models \cite{Aubel12} or learn the matrix directly from the data itself \cite{Tosic11}.

In problem (\ref{vector0}) we are searching for objects within one given measurement $d$. The problem we are discussing in this paper involves several measurements $d_1,\ldots,d_N\in\mathbb{R}^{M_1}$ where each measurement contains the same objects. This is known as multiple measurement problem and can be modeled as
\begin{align}\label{matrix0}
	\min\limits_{X\in\mathbb{R}^{M_2\times N}}\|AX-D\|_F^2+\mu\|X\|_{0,?}
\end{align}
Here, $D=[d_1,\ldots,d_N]\in\mathbb{R}^{M_1\times N}$ is the matrix of all measurements. At first glance problems (\ref{vector0}) and (\ref{matrix0}) seem very similar. However, as $X\in\mathbb{R}^{M_2\times N}$ is now a matrix, the definition of sparsity is not as obvious as it was for vectors. For this reason, we added a question mark to the notation in (\ref{matrix0}). Another problem arising in this multiple measurement setup is, that the objects usually change from one measurement to another. Exemplary, cars or pedestrians in a video change their position and angle to the camera. Seismic waves can change their frequency or wave form depending on the underlying material. Storms in weather data can move or change their intensity. Thus, choosing a suitable sparsity measure $\|X\|_{0,?}$ that can capture these changes is a difficult task.

Many of the discussed approaches from the single measurement case have generalizations in the multiple measurement case. Instead of using the simple $0$-norm, we can count the number of non-zero columns (row sparsity) \cite{Shukla15}, non-zero blocks (block sparsity) \cite{Eldar10}, or more general non-zero groups (group sparsity) \cite{Huang10} in $X$. Again, the problem is NP-hard and requires relaxation \cite{Tropp06_convexRelax} or the use of Greedy methods \cite{Tropp06_Greedy}. Also in the matrix case a sparsity promoting transform $\Phi$ can be used. There are several transforms specialized on two dimensional data such as Shearlets \cite{Kutyniok12}, Curvelets \cite{Ma10}, or a combination of different frames \cite{Kutyniok12_waveletShearlet}. Furthermore, we can apply dictionary learning methods again \cite{Chen17}. Another approach, which does not have an equivalent in the single measurement case, is to assume that the matrix $X$ has a small rank. This means, it only has a few non-zero singular values, i.e., its vector of eigenvalues is sparse \cite{Markovsky12}.

While the mentioned methods perform well in many applications, the sparsity models struggle catching the movement of an object. For this, more sophisticated models are required. The structural sparsity model presented in \cite{Bossmann21_StructSparse} generalizes some of the above mentioned sparsity norms for matrices and allows for more general changes of the object throughout the measurements. Alternatively, shift invariant dictionaries can be used to represent the same object independent of its position \cite{Rusu13}. Similar to this approach, the authors introduced a sparsity model that is based on a shift operator in combination with a rank-$1$ matrix \cite{Bossmann20_SR1}. This idea was later on generalized to the ORKA algorithm (Object reconstruction using $K$-approximation) \cite{Bossmann22_ORKA}. The ORKA algorithm performed well in numerical experiments and various applications. However, a big drawback is that its runtime as well as its ability to accurately track the object movement both depend on the data resolution. We overcome this drawback by introducing an iterative approach: the fast and gridless ORKA. The basic idea of this approach was first presented in a short conference publication \cite{Bossmann23_multiORKA}. In this work, we present a more general discussion and much more detailed analysis.

The remainder of this work is organized as follows. In the next section we shortly introduce the original ORKA algorithm that contains all necessary information to understand the new iterative approach. Knowledge of the previous work \cite{Bossmann22_ORKA} is thus not required. We also discuss the resolution dependency of the original approach and why this is a drawback that needs to be overcome. The third section discusses the new iterative approach. Here, we first introduce the algorithm, and afterwards perform a complexity and error analysis. The new approach is based on resampling the given data to obtain different levels of resolution. One of the resampling strategies is, to minimize the approximation error obtained in the error analysis. This resampling strategy is discussed in Section $4$. Finally, in Section $5$ we present several numerical experiments to verify the theoretical results and demonstrate the algorithm on different applications.

\section{ORKA algorithm}

In this section we will give a brief summary on the ORKA algorithm that will contain all necessary details needed to understand the extension proposed in this work. For more details we refer to the original work \cite{Bossmann22_ORKA}.

The ORKA algorithm is designed to find moving and deforming ``objects'' in data from multiple measurements. The object model used is kept quite general and thus fits to many applications. It can e.g., model seismic wavefronts in geophysical data, walking people in videos, or rainstorm clouds in weather recordings. To model the movement of such objects, the following shift operator is used.

\begin{definition}\label{def:shiftOp}
	Define the matrix $J_M\in\mathbb{R}^{M\times M}$ as
	\begin{align*}
		\bm{J}_M=\begin{pmatrix}
			0 & \cdots & 0 & 1 \\
			1 & \ddots & \vdots & 0 \\
			0 & \ddots & \ddots & \vdots \\
			\ddots & 0 & 1 & 0
		\end{pmatrix}.
	\end{align*}
	Now, for $\lambda\in\mathbb{Z}^N$ we define $S_\lambda:\mathbb{R}^{M\times N}\rightarrow\mathbb{R}^{M\times N}$ as the column shift operator, that shifts the $k$-th column of given data by $\lambda_k$, i.e., for given data $D=[D_{:1},\ldots,D_{:N}]\in\mathbb{R}^{M\times N}$ we have
	\begin{align*}
		S_\lambda(D)=\left[\bm{J}_M^{\lambda_k}D_{:k}\right]_{k=1}^N.
	\end{align*}
	Note that $\bm{J}_M^{-1}=\bm{J}_M^T$ and thus the operator is well-defined even for $\lambda_k<0$. 
\end{definition}

We can use this shift operator to model the movement of simple objects. For example, consider a seismic wave $u\in\mathbb{R}^M$ as signal over time. Assume this wave was observed at $N$ different sensors, with a different amplitude $v_k$ and arrival time $\lambda_k$ for each sensor $k=1,\ldots,N$. This data can be modeled as $S_{\lambda}(uv^T)$ where $v=[v_k]_{k=1}^N$. This model is called the shifted rank-1 matrix and was introduced by the authors in \cite{Bossmann20_SR1}. For ORKA, we will replace the rank-1 matrix $uv^T$ by another model later.

The operator of Definition \ref{def:shiftOp} can only model movement in one dimension. If the data is multi-dimensional we need to generalize the idea. For example, an object in a video can move in the two dimension captured by each frame. In this case we have three-dimensional video data $D\in\mathbb{R}^{M_1\times M_2\times N}$ and use a shift matrix $\lambda\in\mathbb{Z}^{N\times 2}$. The shift operator is then defined as
\begin{align*}
	S_\lambda(D)=\left[\bm{J}_{M_1}^{\lambda_{k,1}}D_{::k}\bm{J}_{M_2}^{-\lambda_{k,2}}\right]_{k=1}^N,
\end{align*}
where $D_{::k}$ is the $k$-th frame of the video. Generally speaking, if the recorded data of each of the $N$ given measurements is $d$ dimensional, then we will use a matrix $\lambda\in\mathbb{Z}^{N\times d}$. The entry $\lambda_{k,j}$ gives the shift of the $k$-th measurement in the $j$-th dimension. For simplicity, we will stick to the one-dimensional case throughout most parts of this work, and only discuss the higher-dimensional cases whenever there is a significant difference.

With the shift operator given in Definition \ref{def:shiftOp}, we can now introduce the object reconstruction problem. Given some data $D$, the ORKA algorithm seeks a moving and deforming object within the data by solving
\begin{align}\label{ORKAopt}
	&\min\limits_{\lambda,U}\left\|D-S_\lambda(U)\right\|_F^2+\mu\sum\limits_{k=2}^N\left\|U_{:k}-U_{:(k-1)}\right\|_2^2,\\
	&s.t.\ \ |\lambda_k-\lambda_{k-1}|\leq C\text{ for }k=2,\ldots,N.\notag
\end{align}
Here, $\|D-S_\lambda(U)\|_F^2$ is the data fidelity term and $\|U_{:k}-U_{:(k-1)}\|_2^2$ is a penalty term which measures the difference (deformation) of the object from one measurement to the next. The parameter $\mu>0$ can be used to control this deformation. The larger $\mu$ is chosen, the smaller the deformation of the object will be. Furthermore, we can set a parameter $C\in\mathbb{N}$ which limits the movement of the object by limiting the change in position of two consecutive measurements. More on the choice of this parameter later. For higher-dimensional data the norms in (\ref{ORKAopt}) will be replaced by the according Frobenius tensor/matrix norm (i.e., the square root of the sum of squares), the absolute value can be replaced by any norm that is adequate for the application to measure the object movement.

The ORKA algorithm solves problem (\ref{ORKAopt}) by applying two steps. First, we replace the data fidelity term by
\begin{align*}
	\left\|D-S_\lambda(U)\right\|_F^2=\left\|S_{-\lambda}(D-S_\lambda(U))\right\|_F^2=\left\|S_{-\lambda}(D)-U\right\|_F^2.
\end{align*}
This separates the variables $\lambda$ and $U$. Now, for a fixed $\lambda$ (\ref{ORKAopt}) becomes a quadratic, convex optimization problem in $U$. The minimum can be calculated analytically:
\begin{align}\label{optProb}
	-\left\langle A^{-1} , S_{-\lambda}(D)(S_{-\lambda}(D))^T \right\rangle.
\end{align}
Here, $A^{-1}$ is the inverse of the system matrix of the quadratic system. To reconstruct $\lambda$, the ORKA algorithm seeks the minimum of (\ref{optProb}) over $\lambda$. However, this is an integer optimization problem and too hard to be solved directly. Instead, we replace $A^{-1}$ by its $K$-bandlimited approximation $A^{-1,[K]}$ defined as
\begin{align}\label{KbandlimitedA}
	A^{-1,[K]}_{j,k}=\begin{cases}
		A_{j,k}^{-1} & |j-k|\leq K \\ 0 & \text{otherwise}
	\end{cases}.
\end{align}
We use the banded structure of the matrix to reduce the problem size and solve for $\lambda$. For the details of this step we refer again to the original work \cite{Bossmann22_ORKA}. For this work, it is enough to describe the basic concept of this step: the problem can be rewritten as a shortest path problem on a graph, which we call the $K$-approximation graph. The size of this graph grows exponentially with $K$, i.e., the complexity of this step is $O((2C+1)^{Kd})$ ($d$ is the dimensionality of the measurement data). On the other hand, the entries of the inverse system matrix $A^{-1}$ are decreasing exponentially away from the diagonal, which means that the approximation error done by ORKA decreases exponentially with $K$. Hence, $K$ needs to be chosen carefully to balance complexity and approximation error.

\subsection{Resolution dependency of ORKA}

The ORKA algorithm as described in the last section is highly dependent on the resolution of the input data in two ways. The first dependency is due to the shift $\lambda\in\mathbb{Z}^N$ being an integer vector, i.e., the shift operator from Definition \ref{def:shiftOp} can only model integer shifts. In other words, the movement of the object is bounded to the grid defined by the data resolution. Second, the parameter $C$ imposes a restriction to the object movement via the constraints in (\ref{ORKAopt}). This restriction is usually given by the application, e.g., a physical limitation on the objects movement speed. For this reason, $C$ can not be chosen to our liking, but typically scales with the resolution.

We demonstrate both dependencies in the following example. Consider a video recording of a running person. We use a static camera that shows a $100$m long street and a person running along this street from the left end of the frame towards the right end. We assume that the person needs at least $10$ seconds for this distance (which is about the current world record). Furthermore, let the video have $100$ frames per second.

First, we demonstrate the dependency of $C$ on the resolution. Here, $C$ will restrict the running speed of the person to less than $10m/s$ (i.e., at least $10$s for $100$m). This is equivalent to $0.1$m per frame. Now let the horizontal resolution of the video be $1000$ pixels. So, the $100$m long street is divided into $1000$ pixel, which means each pixel represents a length of $0.1$m. Hence, we can choose $C=1$ to restrict the movement of the person to one pixel per frame which exactly aligns with the desired maximum speed of $10m/s$. However, if we increase the video resolution to $10,000$ pixels, each pixel will only represent a $0.01$m long part of the road. To model a maximum speed of $0.1$m per frame we now need to set $C=10$. Remember that the complexity of ORKA scales exponentially with base $2C+1$ and thus a small increase in $C$ can have enormous effects on the performance.

Next, let us have a look at the resolution dependency of $\lambda\in\mathbb{Z}^N$. Assume the person actually runs much slower at a speed of only $5m/s$ ($0.05$m per frame). As seen above, a horizontal resolution of $1000$ pixels is equivalent to $0.1$m per pixel. This resolution is too low to catch the persons movement in each frame as we would need to set $\lambda_k=0.5\not\in\mathbb{Z}$. Here, a resolution of at least $2000$ pixels is required to obtain a pixel length of $0.05$m or less. If the person is not running at a constant speed, then even higher resolutions will be required to catch all details of the movement.

In summary, to track the object movement in detail, ORKA requires high resolution data since the movement vector $\lambda$ is bounded to the grid. However, this will also increase the parameter $C$ and thus the complexity. Hence, the data resolution has to be chosen carefully to balance accuracy and complexity. We will present an iterative ORKA approach in the next section that overcomes these problems - the fast and gridless ORKA algorithm (fg-ORKA).

\section{fast gridless ORKA algorithm}

The idea of an iterative ORKA approach was first presented by the authors in a short conference work \cite{Bossmann23_multiORKA}. In that work, we combined ORKA with a wavelet multiresolution analysis to obtain a fast iterative version. The wavelet transform was used to down- or upsample the data by a factor of $2$ and acquire any desired resolution this way. In the here presented work, we generalize the concept to other techniques that allow a down- or upsampling factor different than $2$, which can decrease the runtime even more as we will later see. Furthermore, we will provide a detailed runtime and error analysis for the different methods and provide advise which method to best use depending on the application.

\subsection{Algorithm concept}

The general idea of fg-ORKA is as follows. We start at a low resolution version of the given data where the parameter $C$ is small, hence keeping the complexity low. The obtained movement $\lambda\in\mathbb{Z}^N$ will then be used as approximation for the next higher resolution. This means, for the next higher resolution we do not need to compute the complete movement, but just an update of the low resolution version. We will show that this update can be calculated by using the ORKA algorithm again with a small parameter $C$. This process is repeated until we reach the desired resolution, which can be higher than the original data if we artificially upsample it. This way, we can achieve any desired accuracy on the movement $\lambda$ and are no longer bound to the grid given by the original data resolution.

To resample the data into different resolutions, we need suitable resampling methods with certain properties:

\begin{definition}\label{def:resamplingPair}
	Let $r,M\in\mathbb{N}$ where $r\geq2$ divides $M$. We say that $\Rdown:\mathbb{R}^M\rightarrow\mathbb{R}^{M/r}$ and $\Rup:\mathbb{R}^{M/r}\rightarrow\mathbb{R}^M$ are a suitable resampling pair if the following two conditions hold:
	\begin{itemize}
		\item Both $\Rup$ and $\Rdown$ are linear.
		\item Both operators are $r$-shift invariant in the following sense: For any $x\in\mathbb{R}^M$, $y\in\mathbb{R}^{M/r}$ and $\lambda\in\mathbb{Z}$ we have
		\begin{align*}
			S_\lambda(\Rdown(x))=\Rdown(S_{r\lambda}(x)) && \text{and} && S_{r\lambda}(\Rup(y))=\Rup(S_\lambda(y)).
		\end{align*}
		\item $\Rup\circ\Rdown:\mathbb{R}^M\rightarrow\mathbb{R}^M$ is an orthogonal projection.
		\item $\Rup$ is angle preserving, i.e., for $x,y\in\mathbb{R}^{M/r}$ we have
		\begin{align*}
			\langle x,y\rangle = \langle\Rup(x),\Rup(y)\rangle,
		\end{align*}
		where $\langle\cdot,\cdot\rangle$ is the Euclidean inner product.
	\end{itemize}
	We call $r$ the resampling factor.
\end{definition}

From the listed properties, actually only the $r$-shift invariance is required for the algorithm. We restrict ourselves to linear operators as this will be the common case in applications and the linearity massively simplifies the theory. Last, if the operators do not satisfy the orthogonal projection or angle preserving property, the error bounds achieved later on will be worse. For any operator pair $(\Rup,\Rdown)$ fulfilling Definition \ref{def:resamplingPair} there is an easy representation using a matrix.

\begin{corollary}\label{samplingMatrix}
	Let $r$, $M$, $\Rup$, and $\Rdown$ be as in Definition \ref{def:resamplingPair}. Then there exists an $\rho\in\mathbb{R}^M$ such that
	\begin{align*}
		R=\left(\rho_{j-rk\mod M}\right)_{j,k=0}^{M,\frac{M}{r}}\in\mathbb{R}^{M\times\frac{M}{r}}, &&
		R^TR=\mathbcal{I}_{\frac{M}{r}}\\
		\Rup(x)=Rx, &&
		\Rdown(x)=R^Tx.
	\end{align*}
	where $\mathbcal{I}_{\frac{M}{r}}$ is the identity matrix of size $\frac{M}{r}\times\frac{M}{r}$.
\end{corollary}
\begin{proof}
	Since $\Rup$ is linear, there exists a matrix such that $\Rup(x)=Rx$ for all $x\in\mathbb{R}^{\frac{M}{r}}$. From the angle preserving property we get
	$$x^Ty=\langle x,y\rangle=\langle\Rup(x),\Rup(y)\rangle=x^TR^TRy,$$
	for all $x,y\in\mathbb{R}^{\frac{M}{r}}$. This can only hold if $R^TR=\mathbcal{I}_{\frac{M}{r}}$. This also means that the columns for $R$ form an orthogonal basis of some subspace of $\mathbb{R}^M$. $\Rup\circ\Rdown$ is the orthogonal projection onto this subspace, which can also be written in forms of the Moore-Penrose inverse, i.e., $\Rup\circ\Rdown=RR^+$. Because the orthogonal projection is unique, we obtain $$\Rdown(x)=R^+x=(R^TR)^{-1}R^Tx=R^Tx.$$
	Last, we use the $r$-shift invariance of our operators. Therefore, we rewrite the equation in terms of matrix multiplications using Definition \ref{def:shiftOp}:
	\begin{align*}
		S_\lambda(\Rdown(x))=\Rdown(S_{r\lambda}(x)) && \Leftrightarrow && \bm{J}_{\frac{M}{r}}^\lambda R^T x=R^T\bm{J}_M^{r\lambda}x \\
		S_{r\lambda}(\Rup(x))=\Rup(S_{\lambda}(x)) && \Leftrightarrow && \bm{J}_M^{r\lambda} R x=R\bm{J}_{\frac{M}{r}}^{\lambda}x \\
	\end{align*}
	As this needs to hold for all $x$ and $\lambda$, we require $\bm{J}_M^{r\lambda} R=R\bm{J}_{\frac{M}{r}}^{\lambda}$. (Remember that $\bm{J}_M^{-\lambda}=(\bm{J}_M^\lambda)^T$ and thus the shift invariance of the downsampling operator also follows from this requirement.) For $j=0,\ldots,M-1$ and $k=0,\ldots,\frac{M}{r}-1$ we get
	\begin{align*}
		\left(\bm{J}_M^{r\lambda} R\right)_{j,k}=\left(R\bm{J}_{\frac{M}{r}}^{\lambda}\right)_{j,k} &&
		\Leftrightarrow &&
		R_{j-r\lambda\mod M,k}=R_{j,k+\lambda\mod\frac{M}{r}}.
	\end{align*}
	Now choose $k=0$ and $\lambda=1,\ldots,\frac{M}{r}-1$ and we see that each column of the matrix is just a shifted version of the first column, i.e., we can set $\rho_j=R_{j,0}$ for $j=0,\ldots,M-1$.
\end{proof}

We will use the matrix $R$, the vector $\rho$, and the operators $\Rup$, $\Rdown$ interchangeably throughout this work as they all represent the same resampling methods. In this work we focus on three different strategies. The first two are straightforward and given in the following two examples. The third resampling strategy tries to minimize the approximation error of our method. It is presented after the error analysis of fg-ORKA was discussed.

\begin{example}
	Let $M$ be divisible by $2$. The periodic discrete wavelet transform divides a given signal $x\in\mathbb{R}^M$ into its low-pass $x_{\text{low}}\in\mathbb{R}^{M/2}$ and high-pass coefficients $x_{\text{high}}\in\mathbb{R}^{M/2}$. The high-pass coefficients contain the details of the signal while the low-pass coefficients can be seen as a low-resolution approximation. We use these to define our downsampling operator $\Rdown(x)=x_{\text{low}}$. This is equivalent to setting $\rho$ to the scaling coefficients associated with the Wavelet. To get the angle preserving property we need to choose an orthogonal Wavelet such as the Daubechies wavelet family. The upsampling operator $\Rup$ will in this case perform an inverse Wavelet transform where the data is used as low-pass coefficients and the high-pass coefficients are assumed to be $0$.
\end{example}

The above example was the strategy used in our first work \cite{Bossmann23_multiORKA}. The drawback on wavelet based resampling is, that the sampling rate is fixed to $r=2$. We show in our later analysis that this is not always desirable. Thus, we also provide the following simple strategy that works with any resampling factor.

\begin{example}
	Let $F_M\in\mathbb{C}^{M\times M}$ be the discrete normalized Fourier matrix defined as
	\begin{align*}
		F_M=\frac{1}{\sqrt{M}}\left(
			e^{\frac{-2\pi ijk}{N}}
		\right)_{j,k=0}^{M-1}.
	\end{align*} 
	The Fourier transform of a vector $x\in\mathbb{C}^M$ is defined as $\hat{x}=F_Mx$. The matrix is orthogonal and thus the inverse transform is given as $F_M^*\hat{x}=x$. Furthermore, for real vectors $x\in\mathbb{R}^M$ we have $\hat{x}_j=\overline{\hat{x}_{M-j}}$ for $j=1,\ldots,M-1$. For any $r$ that divides $M$, we can define a Fourier based downsampling operator as
	\begin{align*}
		\Rdown(x)=F_{\frac{M}{r}}^*PF_Mx
	\end{align*}
	where $P\in\mathbb{R}^{(M/r)\times M}$ is a downsampling matrix of the form
	\begin{align*}
		P&=
		\begin{pmatrix}
			\mathbcal{I}_{\lceil M/(2r)\rceil} & \bm{0} & \bm{0} \\ \bm{0} & \bm{0} & \mathbcal{I}_{\lfloor M/(2r)\rfloor}
		\end{pmatrix}&\text{ for }\frac{M}{r}\text{ odd},\\
		P&=
		\begin{pmatrix}
			\mathbcal{I}_{M/(2r)} & \bm{0} & \bm{0} & \bm{0} & \bm{0} \\ \bm{0} & \frac{1}{\sqrt{2}} & \bm{0} &  \frac{1}{\sqrt{2}} & \bm{0} \\ \bm{0} & \bm{0} & \bm{0} & \bm{0} & \mathbcal{I}_{M/(2r)-1}
		\end{pmatrix}&\text{ for }\frac{M}{r}\text{ even},
	\end{align*}
	i.e., we remove the high-frequency coefficients in the middle of the data. The downsampling operator $\Rdown$ applies a length $M$ Fourier transform, followed by a low-pass filtering, and last a length $\frac{M}{r}$ inverse Fourier transform. We have $\Rdown(x):\mathbb{R}^M\rightarrow\mathbb{R}^{\frac{M}{r}}$ and thus $F_{\frac{M}{r}}PF_M\in\mathbb{R}^{\frac{M}{r}\times M}$. It follows that 
	\begin{align*}
		(F_{\frac{M}{r}}PF_M)(F_{\frac{M}{r}}PF_M)^T=(F_{\frac{M}{r}}PF_M)(F_{\frac{M}{r}}PF_M)^*=\mathbcal{I}_{\frac{M}{r}}.
	\end{align*}
	Furthermore, the $r$-shift invariance is a direct consequence of the Fourier shift theorem. Hence, $\Rdown$ and $\Rup$ are suitable resampling operators following from Corollary \ref{samplingMatrix}.
\end{example}

Let us now discuss the fg-ORKA algorithm in detail. Consider the following setup. For given data $D\in\mathbb{R}^{M\times N}$ and $C>0$ we seek to find the optimal movement vector $\lambda^{\text{opt}}\in\mathbb{Z}^N$ as solution of problem (\ref{ORKAopt}). The fg-ORKA approach assumes that, if we downsample the data by a factor of $r$, the optimal path $\lambda^{\text{opt},1}$ of the downsampled data $\Rdown(D)$ is an approximation of the original optimal path. In concrete, we assume that the relative distances which appear in the constraint of problem (\ref{ORKAopt}) are preserved as best as possible:
\begin{align}\label{distancePreserving}
	\left|\lambda^{\text{opt},1}_k-\lambda^{\text{opt},1}_{k-1}\right|=\text{round}\left(\frac{\left|\lambda^{\text{opt}}_k-\lambda^{\text{opt}}_{k-1}\right|}{r}\right), && k=2,\ldots,N.
\end{align}
Using the $r$-shift invariance, we can upscale $\lambda^{\text{opt},1}$ and write $\lambda^{\text{opt}}=r\lambda^{\text{opt},1}+\lambda^{\text{diff}}$ where $\lambda^{\text{diff}}\in\mathbb{Z}^N$ is the difference due to the rounding effect in (\ref{distancePreserving}). We get
\begin{align*}
	\text{round}\left(\frac{\left|\lambda^{\text{opt}}_k-\lambda^{\text{opt}}_{k-1}\right|}{r}\right)
	&=\text{round}\left(\left|\lambda^{\text{opt},1}_k-\lambda^{\text{opt},1}_{k-1}+\frac{\lambda^{\text{diff}}_k-\lambda^{\text{diff}}_{k-1}}{r}\right|\right)
\end{align*}
Since $\lambda^{\text{opt},1}_k-\lambda^{\text{opt},1}_{k-1}\in\mathbb{Z}$ we get with (\ref{distancePreserving}) that
\begin{align}\label{updateConstant}
	\left|\frac{\lambda^{\text{diff}}_k-\lambda^{\text{diff}}_{k-1}}{r}\right|\leq\frac{1}{2}
	&&\Leftrightarrow&&
	\left|\lambda^{\text{diff}}_k-\lambda^{\text{diff}}_{k-1}\right|\leq\frac{r}{2}.
\end{align}
Last, note that $S_{\lambda^{\text{opt}}}(D)=S_{\lambda^{\text{diff}}}(S_{r\lambda^{\text{opt},1}}(D))$. This inspires the following strategy: First, use the lower resolution data $\Rdown(D)$ to reconstruct $\lambda^{\text{opt},1}$ with the ORKA algorithm. Second, use the ORKA algorithm again on the pre-shifted data $S_{r\lambda^{\text{opt},1}}(D)$ with $C=\frac{r}{2}$ to obtain $\lambda^{\text{diff}}$. This strategy can be repeated to recover $\lambda^{\text{opt},1}=r\lambda^{\text{opt},2}+\lambda^{\text{diff},1}$ in two steps. Note that with each further downsampling step added, the constant $C$ required for the initial ORKA algorithm is also divided by $r$ and thus the complexity is reduced.

Furthermore, we can apply a similar strategy in the other direction. Assume we have reconstructed the path vector $\lambda^{\text{opt}}$. Set $\lambda^{\text{opt},0}=\lambda^{\text{opt}}$. We can now calculate the paths $\lambda^{\text{opt},-1}=r\lambda^{\text{opt},0}+\lambda^{\text{diff},-1}$ for the artificially upsampled data $\Rup(D)$ in the same manner. While the artificial upsampling does not add any new information about the data, it overcomes the grid dependency of our path. We can use the new optimal path $\lambda^{\text{opt}}=\frac{\lambda^{\text{opt},-1}}{r}\in\left\{\frac{k}{r}\ \middle|\ k\in\mathbb{Z}\right\}^N$ scaled to the original resolution of $D$ that is no longer bounded to the grid. Again, this idea can be applied several times to achieve a more detailed path vector.

Let $L\in\mathbb{N}$ be the number of downsamples performed in our strategy and $J\in\mathbb{N}$ the number of artificial upsamples. We can write the optimal path vector as
\begin{align}\label{lambdaOptRepresent}
	\lambda^{\text{opt}}=r^L\lambda^{\text{opt},L}+\sum\limits_{j=-J}^{L-1}r^j\lambda^{\text{diff},j}
	\in\left\{\frac{k}{r^J}\ \middle|\ k\in\mathbb{Z}\right\}^N
	&&
	\text{with}
	&&
	|\lambda^{\text{diff},j}_k-\lambda^{\text{diff},j}_{k-1}|\leq\left\lfloor\frac{r}{2}\right\rfloor,
\end{align}
where we can apply the floor operator on $\frac{r}{2}$ since $\lambda^{\text{diff},j}\in\mathbb{Z}$. To balance the complexity of all ORKA calls, we also require $|\lambda^{\text{opt},L}_k-\lambda^{\text{opt},L}_{k-1}|\leq\left\lfloor\frac{r}{2}\right\rfloor$. Note that problem (\ref{ORKAopt}) requires $|\lambda^{\text{opt}}_k-\lambda^{\text{opt}}_{k-1}|\leq C$ and we want this bound to be tight in order to not restrict the number of possible paths further. This determines $L$ by
\begin{align}
	|\lambda^{\text{opt}}_k-\lambda^{\text{opt}}_{k-1}|
	\leq
	r^L\left\lfloor\frac{r}{2}\right\rfloor+\sum\limits_{j=-J}^{L-1}r^j\left\lfloor\frac{r}{2}\right\rfloor
	=
	\left\lfloor\frac{r}{2}\right\rfloor\frac{r^{L+1}-r^{-J}}{r-1}
	\stackrel{!}{=}C\label{LCconnection}\\
	\Rightarrow
	L=\log_r\left(\frac{C(r-1)}{\left\lfloor\frac{r}{2}\right\rfloor}+r^{-J}\right)-1=\begin{cases}
		\log_r\left(2C+r^{-J}\right)-1 & ,r\text{ odd}\\
		\log_r\left(2C(1-r^{-1})+r^{-J}\right)-1 & ,r\text{ even}
	\end{cases},\label{chooseL}
\end{align}
where we round the result either up or down to the next natural number. (Rounding up results in the constraint of problem (\ref{ORKAopt}) being slightly violated in the extreme cases, while rounding down will restrict our object movement slightly more than intended.) Altogether, we obtain the the fg-ORKA algorithm as shown in Algorithm \ref{alg:fg-ORKA}.
\begin{algorithm}[h]
	\caption{fg-ORKA}
	\label{alg:fg-ORKA}
	\DontPrintSemicolon
	
	\SetKwInput{Input}{Input}
	\SetKwInOut{Output}{Output}
	
	\Input{data $D\in\mathbb{R}^{M\times N}$, ORKA parameters $C,\mu\in\mathbb{R}_+$,\newline resampling rate $r$, resampling pair $(\Rdown,\Rup)$,\newline number of upsamplings $J\in\mathbb{N}$.}
	\tcc{data resampling}
	Calculate $L$ according to (\ref{chooseL}) and set $D^0=D$\;
	\lFor{$k=1,\ldots,L$}{Calculate $D^k=\Rdown(D^{k-1})$}
	\lFor{k=-1,\ldots,-J}{Calculate $D^k=\Rup(D^{k+1})$}
	\tcc{iterative ORKA calls}
	Calculate $\lambda^{\text{opt},L}$ using ORKA (on $D^L$ with $C=\left\lfloor\frac{r}{2}\right\rfloor$)\;
	\For{k=L-1,\ldots,-J+1}{
		Calculate $\lambda^{\text{diff},k}$ using ORKA (on $S_{r\lambda^{\text{opt},{k+1}}}(D^k)$ with $C=\left\lfloor\frac{r}{2}\right\rfloor$)\;
		Calculate $\lambda^{\text{opt},k}=r\lambda^{\text{opt},{k+1}}+\lambda^{\text{diff},k}$\;
	}
	Calculate $\lambda^{\text{diff},{-J}}$ and $U$ using ORKA (on $S_{r\lambda^{\text{opt},{-J+1}}}(D^{-J})$ with $C=\left\lfloor\frac{r}{2}\right\rfloor$)\;
	Calculate $\lambda^{\text{opt},{-J}}=r\lambda^{\text{opt},{-J+1}}+\lambda^{\text{diff},{-J}}$\;
	Revert pre-shift $U\leftarrow S_{-r\lambda^{\text{opt},{-J+1}}}(U)$\;
	\Output{$\lambda^{\text{opt}}=\frac{\lambda^{\text{opt},{-J}}}{r^J}$, $U\in\mathbb{R}^{r^JM\times N}$}
\end{algorithm}

\begin{remark}
	There are a few details on Algorithm \ref{alg:fg-ORKA} that we want to point out. First, because of the downsampling strategy, $M$ needs to be divisible by $r^L$. This can be achieved by zero-padding the data if required. Second, the original ORKA algorithm is used in lines 4 and 6 only to obtain the path update, i.e., only the first step of ORKA is actually required and we do not need to solve for $U$. Only in the last step (line 8) we also return the object matrix $U$. Note that because of the pre-shift applied to the data, we need to shift $U$ back into the position that corresponds to an unshifted $D^{-J}$ (line 10). Also, note that $U\in\mathbb{R}^{r^JM\times N}$ has a much finer resolution than the original data. Last, for multidimensional data $D\in\mathbb{R}^{M_1\times\ldots\times M_m\times N}$ the algorithm can be applied in the same way as long as suitable resampling functions $\Rdown:\mathbb{R}^{M_1\times\ldots\times M_m}\rightarrow\mathbb{R}^{M_1/r\times\ldots\times M_m/r}$ and $\Rup:\mathbb{R}^{M_1/r\times\ldots\times M_m/r}\rightarrow\mathbb{R}^{M_1\times\ldots\times M_m}$ are given.
\end{remark}

In the next two subsections we analyze the complexity of the algorithm as well as the approximation error compared to the original ORKA algorithm. Afterwards, we present a third downsampling strategy that is based on the idea of minimizing the approximation error.

\subsection{Complexity analysis}

For our complexity analysis we concentrate on the path reconstruction using the first step of the ORKA algorithm, i.e., lines 4, 6, and 8 of Algorithm \ref{alg:fg-ORKA}. This is the part of our algorithm that scales exponentially. Other steps, such as the resampling of the data (lines 2,3), only have a very minor effect on the complexity that is negligible in comparison. As a reminder, the complexity in both runtime and memory usage of the original ORKA algorithm is $O((2C+1)^{Kd})$ where $K$ is the approximation parameter and $d$ is the dimensionality of one measurement. Since fg-ORKA perform $L+J+1$ calls to the original algorithm with $C=\left\lfloor\frac{r}{2}\right\rfloor$, it directly follows that fg-ORKA has a space complexity of
\begin{align*}
	O\left((2\left\lfloor\frac{r}{2}\right\rfloor+1)^{Kd}\right),
\end{align*}
which is the space complexity of the ORKA algorithm in each iteration. Remember that $K$ is the number of bands used in our bandlimited approximation matrix (\ref{KbandlimitedA}), i.e., the larger we choose $K$ the smaller the approximation error gets. For this reason, the usual parameter strategy is, to choose $r=2,3$ and then set $K$ as large as possible until the available memory is exhausted. Nevertheless, we also want to analyze how the runtime complexity of fg-ORKA is effected by the choice of $r$ and $K$. The runtime complexity of fg-ORKA is the runtime complexity of ORKA multiplied with the number of iterations performed, i.e.,
\begin{align}\label{complexity}
	O\left((L+J+1)(2\left\lfloor\frac{r}{2}\right\rfloor+1)^{Kd}\right).
\end{align}
For large $C$ this can be much more efficient than the original algorithm since the base of the exponential can be reduced drastically, i.e., $\left\lfloor\frac{r}{2}\right\rfloor\ll C$. The new method only scales linear in $L$ and $J$, but there are some details that we need to take into account. First, the number $J$ of upsamplings performed also influences the final data size (see output size of $U$ in Algorithm \ref{alg:fg-ORKA}). This will increase the complexity of several other steps of the algorithm, e.g., the convex optimization done to recover $U$. This needs to be considered whenever using a large parameter $J$. However, as we will see in the experiments later on, the benefit of artificially upscaling the data diminishes after a few steps and thus $J$ is typically quite small. Also, note that the number of ORKA calls in fg-ORKA using the parameters $C$ and $J$ is the same as using the parameters $\tilde{C}=r^JC$ and $\tilde{J}=0$, i.e., the complexity does not change if we upsample the data $J$ times beforehand and then set $J=0$. Without loss of generality, we will use $J=0$ in the remaining analysis.

We compare the complexity of fg-ORKA for different resampling factors $r$. Remember that due to (\ref{chooseL}) $L$ depends on $r$ and $C$. Intuitively, when the resampling rate $r$ increases the number of required iterations $L$ should go down. This can best be seen from the sum formula in (\ref{LCconnection}). The expression increases for increasing $r\geq2$, which means $L$ has to be reduced to fit the target value $C$. Since $\left\lfloor\frac{2s}{2}\right\rfloor=\left\lfloor\frac{2s+1}{2}\right\rfloor$, we can directly follow that for even resampling rates $r=2s$ the next higher odd resampling rate $r=2s+1$ has the same or a lower runtime complexity.

Next, we compare the complexity for resampling rates $r$ and $r+2$. For odd $r$ combining (\ref{chooseL}) and (\ref{complexity}) we get a complexity of
\begin{align*}
	O_r=O\left(\text{round}\left(\log_r(2C+1)\right)r^{Kd}\right).
\end{align*}
Note that we need to round $L$ to an integer value as it is the number of iterations performed. Without the rounding operation, we can calculate the derivative to see that the complexity increases for $r\geq3$ and thus $r=3$ should be the optimal (odd) choice. To prove this in more detail, we consider the rate $\frac{O_3}{O_r}$ between the complexity for $r=3$ and any other odd $r\neq3$. Therefore, let $j,k\in\mathbb{N}$ be chosen such that
\begin{align}
	&&&\begin{matrix}j-0.5\leq \log_3(2C+1)< j+0.5 \\ k-0.5\leq \log_r(2C+1) < k+0.5\end{matrix},\notag\\
	\Leftrightarrow&&&
	\begin{matrix}\log3(j-0.5)\leq \log(2C+1)<\log3(j+0.5) \\ \log r(k-0.5)\leq \log(2C+1) <\log r(k+0.5)\end{matrix}.\label{roundingjk}
\end{align}
Note that resampling rates $r>2C+1$ are not feasible, as this increases the complexity compared to the original ORKA algorithm. Thus we assume $r\leq2C+1$ and get $j,k\geq1$. Now, combining both inequalities in (\ref{roundingjk}), we get
\begin{align*}
	\log 3(j-0.5)<\log r(k+0.5) && \Leftrightarrow && j<\frac{\log r}{\log3}(k+0.5)+0.5.
\end{align*} 
From this we obtain
\begin{align*}
	\frac{O_3}{O_r}=\frac{j3^{Kd}}{kr^{Kd}}<\frac{\frac{\log r}{\log3}(k+0.5)+0.5}{k}\left(\frac{3}{r}\right)^{Kd}
	\leq\left(\frac{\log r}{\log 3}+\frac{\log r}{2\log 3}+\frac{1}{2}\right)\left(\frac{3}{r}\right)^{Kd}.
\end{align*}
Next, note that $\frac{\log r}{r}$ decreases for $\log r\geq1$ and hence the maximum for all odd $r\neq3$ is reached at $r=5$. We use that to obtain the bound
\begin{align*}
	\frac{O_3}{O_r}<\left(\frac{3\log 5}{5\log 3}+\frac{3\log 5}{10\log 3}+\frac{3}{10}\right)\left(\frac{3}{r}\right)^{Kd-1}<1.62\left(\frac{3}{r}\right)^{Kd-1}.
\end{align*}
Since $1.62<\frac{5}{3}$ this bound is smaller $1$ for all odd $r\geq5$ and all $Kd\geq2$, i.e., the runtime complexity increases. (The case where $Kd=1$ is not relevant in applications since this parameter choice is not recommended anyway.) In the same way it can be shown that the runtime complexity for even $r$ increases with $r$ for $Kd\geq2$. Altogether, we obtain that $r=3$ is the most efficient resampling rate.

\subsection{Error analysis}

Before we go deeper into the error analysis, we want to clarify some things about the method and its approximation error. First, ORKA and fg-ORKA are no approximation algorithms but designed to track objects within the data. This means, the actual approximation error with respect to the original data $\|D-S_\lambda(U)\|_F^2$ is not relevant to measure the quality of the algorithm. Indeed, we can easily achieve an approximation error of $0$ by choosing $\mu=0$ in (\ref{ORKAopt}). Instead, we analyze how well ORKA and fg-ORKA recover the optimal value given in (\ref{optProb}), i.e., how optimal the reconstructed movement is. For the original ORKA algorithm this error is $O((N-K)^2e^{(N-K)^2})$ \cite{Bossmann22_ORKA} where $N$ is the number of measurements given.

Furthermore, we remind the reader that ORKA is a two-step method. In the first step the movement vector $\lambda$ is reconstructed using a $K$-approximation. The second step calculates the corresponding object matrix $U$. It is important to note, that the second step does not use any approximation but actually solves the exact problem (for the fixed movement $\lambda$). Hence, as long as the optimal path is reconstructed in the first step, the approximation error of ORKA and fg-ORKA will be $0$. Unfortunately, we cannot give an exact analysis on when the optimal path is reconstructed and when not. Instead, the error bounds we give will show which factors play a role in the success or failure of the first step. The approximation errors observed in practice normally tend to stay very small until the path reconstruction fails at which point the error drastically increases.

\begin{remark}\label{rmk:numbersBaseRelation}
	The proposed algorithm reconstructs the movement vector $\lambda$ in several iterations using the summation (\ref{lambdaOptRepresent}). This formula is closely related to the representation of numbers within the base $r$. Instead of having a representation using digits $0,\ldots,r-1$, formula (\ref{lambdaOptRepresent}) uses the digits $-\left\lfloor\frac{r}{2}\right\rfloor,\ldots,\left\lfloor\frac{r}{2}\right\rfloor$. For odd $r$ this is actually a valid numeral system, e.g., for $r=3$ this is called the balanced ternary. It follows, that there is a unique representation for each number using exactly $L+J+1$ digits (allowing leading zeros), i.e., for any optimal path $\lambda^{\text{opt}}$ there is only one possible choice of $\lambda^{\text{opt},L}$ and $\lambda^{\text{diff},j}$. In other words, once fg-ORKA fails reconstructing the correct path in one of its iterations, this error can not be undone in the following steps. If $r$ is even instead, the representation is no longer unique. For example, $5=1\times1+0\times2+1\times4=-1\times1+1\times2+1\times4$ has two different representations for $r=2$. This means, for even $r$, fg-ORKA has the chance of correcting an error in later iterations. For this reason, we consider $r=2$ the more stable resampling rate while $r=3$ is the more efficient one.
\end{remark}

To understand the approximation error done by fg-ORKA, we first need to understand how the downsampling process influences the values of our optimization problem (\ref{optProb}). For our analysis, we assume that $D\in\mathbb{R}^{M\times N}$ has columns with $\|D_{:j}\|_2\leq1$. Let $r$ be the resampling rate and $\lambda=r\lambda^r+\lambda^{\text{diff}}\in\mathbb{Z}^N$ with $|\lambda^{\text{diff}}_k-\lambda^{\text{diff}}_{k-1}|\leq\left\lfloor\frac{r}{2}\right\rfloor$. We are interested in an error bound of
\begin{align}
	&\left|
	\left\langle A^{-1}, (S_\lambda(D))^T S_\lambda(D) \right\rangle -
	\left\langle A^{-1}, (S_{\lambda^r}(\Rdown(D)))^T S_{\lambda^r}(\Rdown(D)) \right\rangle
	\right|\label{fullerror}\\
	\leq&\left|
	\left\langle A^{-1}, (S_\lambda(D))^T S_\lambda(D) \right\rangle -
	\left\langle A^{-1}, (S_{r\lambda^r}(D))^T S_{r\lambda^r}(D) \right\rangle
	\right|\label{differror}\\
	+&\left|
	\left\langle A^{-1}, (S_{r\lambda^r}(D))^T S_{r\lambda^r}(D) \right\rangle -
	\left\langle A^{-1}, (S_{\lambda^r}(\Rdown(D)))^T S_{\lambda^r}(\Rdown(D)) \right\rangle
	\right|,\label{scalingerror}
\end{align}
which is the difference in the optimal value (\ref{optProb}) for the original movement $\lambda$ and the downsampled version $\lambda^r$. We bound (\ref{differror}) by
\begin{align*}
	&\left|
	\left\langle A^{-1}, (S_\lambda(D))^T S_\lambda(D) \right\rangle -
	\left\langle A^{-1}, (S_{r\lambda^r}(D))^T S_{r\lambda^r}(D) \right\rangle
	\right|\\
	\leq&\sum\limits_{j,k}^N\left|A_{jk}^{-1}\right|\left|\left\langle S_{\lambda_j}(D_{:j}) , S_{\lambda_k}(D_{:k}) \right\rangle - \left\langle S_{r\lambda^r_j}(D_{:j}) , S_{r\lambda^r_k}(D_{:k}) \right\rangle\right|\\
	=&\sum\limits_{j,k}^N\left|A_{jk}^{-1}\right|\left|\left\langle D_{:j} , S_{\lambda_k-\lambda_j}(D_{:k})  - S_{r\lambda^r_k-r\lambda^r_j}(D_{:k}) \right\rangle\right|\\
	=&\sum\limits_{j,k}^N\left|A_{jk}^{-1}\right|\left|\left\langle D_{:j} , S_{r\lambda^r_k-s\lambda^r_j}\left(S_{\lambda^{\text{diff}}_k-\lambda^{\text{diff}}_j}(D_{:k})  - D_{:k}\right) \right\rangle\right|\\
	\leq&\sum\limits_{j,k}^N\left|A_{jk}^{-1}\right|\left\|S_{\lambda^{\text{diff}}_k-\lambda^{\text{diff}}_j}(D_{:k})  - D_{:k}\right\|_2\\
\end{align*}
To find a bound for the norm, we use the Fourier transform together with the Fourier shift theorem to get
\begin{align*}
\left\|S_{\lambda^{\text{diff}}_k-\lambda^{\text{diff}}_j}(D_{:k})  - D_{:k}\right\|_2^2
&=
\left\|\text{diag}\left(e^{\frac{-2\pi i l (\lambda^{\text{diff}}_k-\lambda^{\text{diff}}_j)}{M}}\right)_{l=0}^{M-1} F_M(D_{:k})  - F_M(D_{:k})\right\|_2^2
\\
&=
\sum\limits_{l=0}^{M-1}\left|\left(e^{\frac{-2\pi i l (\lambda^{\text{diff}}_k-\lambda^{\text{diff}}_j)}{M}}-1\right) (F_M(D_{:k}))_l\right|^2
\\
&=
\sum\limits_{l=0}^{M-1} 2|(F_M(D_{:k}))_l|^2\left(1-\cos\frac{2\pi l (\lambda^{\text{diff}}_k-\lambda^{\text{diff}}_j)}{M}\right)\\
&=
4\sum\limits_{l=0}^{M-1} \left(|(F_M(D_{:k}))_l|\sin\frac{\pi l (\lambda^{\text{diff}}_k-\lambda^{\text{diff}}_j)}{M}\right)^2\\
\end{align*}
Note that the value of $\sin^2(\ldots)$ is the same for $l=l'$ and $l=M-l'$. Furthermore, we have $|\lambda^{\text{diff}}_k-\lambda^{\text{diff}}_j|\leq|j-k|\left\lfloor\frac{r}{2}\right\rfloor$ and $\sin^2(x)$ is increasing for $x\in[0,\pi/2]$. Last, from \cite{Bossmann22_ORKA} we know that $\left|A_{jk}^{-1}\right|=O(e^{-|j-k|})$, i.e., the coefficients of the inverse matrix decrease exponentially away from the diagonal. Thus, the error (\ref{differror}) scales as
\begin{align}\label{griderror}
O\left(
4\sum\limits_{l=0}^{M-1}|(F_M(D_{:k}))_l|^2 L(j,k,l)\right)
\end{align}
where
\begin{align}\label{Lgrid}
	L(j,k,l)=\begin{cases}e^{-|j-k|}
		\left(\sin\frac{\pi l |j-k|\left\lfloor\frac{r}{2}\right\rfloor}{M}\right)^2 & \text{, if } l |j-k|\left\lfloor\frac{r}{2}\right\rfloor\leq\frac{M}{2}\text{ and }l\leq\frac{M}{2}\\
		L(j,k,M-l) & \text{, if }l>\frac{M}{2}\\
		e^{-|j-k|} & \text{, otherwise}
	\end{cases},
\end{align}
i.e., the error is small when the data is mostly low frequency. Furthermore, decreasing the resampling rate $r$ can also decrease the error.

To find an upper bound for (\ref{scalingerror}), we use the matrix representation of Corollary \ref{samplingMatrix}. We have
\begin{align*}
	(S_{\lambda^r}(\Rdown(D)))^TS_{\lambda^r}(\Rdown(D))
	&=
	(S_{\lambda^r}(R^TD))^TS_{\lambda^r}(R^TD)\\
	&=
	(R^TS_{r\lambda^r}(D))^TR^TS_{r\lambda^r}(D)\\
	&=
	(RR^TS_{r\lambda^r}(D))^TRR^TS_{r\lambda^r}(D)
\end{align*}
Since $RR^T$ is an orthogonal projection, we can write $S_{r\lambda^r}(D)=RR^TS_{r\lambda^r}(D)+X$ where $RR^TS_{r\lambda^r}(D)\perp X$. It follows that
\begin{align*}
(S_{r\lambda^r}(D))^T S_{r\lambda^r}(D)
&=
(RR^TS_{r\lambda^r}(D)+X)^T (RR^TS_{r\lambda^r}(D)+X)\\
&=(RR^TS_{r\lambda^r}(D))^TRR^TS_{r\lambda^r}(D)+X^TX
\end{align*}
Note that $X=S_{r\lambda^r}(D)-RR^TS_{r\lambda^r}(D)=(\mathbcal{I}_M-RR^T)S_{r\lambda^r}(D)$. Altogether, we obtain for the scaling error (\ref{scalingerror})
\begin{align*}
	\left|\left\langle A^{-1},(\mathbcal{I}_M-RR^T)S_{r\lambda^r}(D)  \right\rangle\right|
	&\leq
	\|A^{-1}\|_F\|(\mathbcal{I}_M-RR^T)S_{r\lambda^r}(D)\|_F\\
	&=\|A^{-1}\|_F\|(\mathbcal{I}_M-RR^T)D\|_F
\end{align*}
i.e., the error scales with the approximation error of the orthogonal projection.

To summarize, the overall error (\ref{fullerror}) depends on three factors: the chosen resampling rate $r$, the frequency distribution of the original data, and the approximation error of the orthogonal projection. Note, that this is the error done in one iteration of fg-ORKA. The overall error is given as the sum over all errors for the different iterations. However, we want to remind the reader about the discussion at the beginning of this subsection. As long as fg-ORKA reconstructs the optimal path in one iteration, the approximation error will remain $0$ for this step. This means, looking at each iteration individually can actually tell us more about the success rate than just looking at the overall error.

\section{Optimal resampling}

In this section, we introduce a third resampling strategy that is based on minimizing the error bounds discussed in the last section. We start by minimizing the approximation error of the orthogonal projection, i.e., we are searching for the matrix $R$ that solves
\begin{align*}
	\min\limits_{R}\|D-RR^TD\|_F^2.
\end{align*}
Remember that due to Corollary \ref{samplingMatrix} the columns of the matrix $R$ are shifted versions of a vector $\rho\in\mathbb{R}^M$. Let $X=R^TD$ for now. We can apply a Fourier transform and use the Fourier shift theorem to obtain
\begin{align*}
	\|D-RX\|_F^2&=\|\hat{D}-\hat{R}X\|_F^2=\left\|\hat{D}-\text{diag}(\hat{\rho})\begin{pmatrix}\hat{X} \\ \vdots \\ \hat{X}\end{pmatrix}\right\|_F^2,
\end{align*}
where $\text{diag}(\hat{\rho})$ is a diagonal matrix with $\hat{\rho}$ on its diagonal and $\hat{X}$ is repeated $r$ times in this expression. Denote the $k$-th row of $\hat{D}$ and $\hat{X}$ by $\hat{D}_{k,:}$ and $\hat{X}_{k,:}$ respectively. Then the above term can be rewritten as
\begin{align}\label{rank1approximation}
\left\|\hat{D}-\text{diag}(\hat{\rho})\begin{pmatrix}\hat{X} \\ \vdots \\ \hat{X}\end{pmatrix}\right\|_F^2
&=
\sum\limits_{k=0}^{M/r-1}\left\|
\begin{pmatrix}
	\hat{D}_{k+lM/r,:}
\end{pmatrix}_{l=0}^{r-1}-\begin{pmatrix}
\hat\rho_{k+lM/r}
\end{pmatrix}_{l=0}^{r-1}
\hat{X}_{k,:}
\right\|_F^2.
\end{align}
Here, each summand can be interpreted as a rank-$1$ approximation of parts of the data. It is well known that the best rank-$1$ approximation is given by the largest singular value and its corresponding vectors. Hence, we can set
\begin{align}\label{eigenvector_solution}
	\begin{pmatrix}
	\hat\rho_{k+lM/r}
\end{pmatrix}_{l=0}^{r-1}=\alpha_ku_k
\end{align}
where $\alpha_k\in\mathbb{C}$, $\alpha_k\neq0$ and $u_k$ is the corresponding eigenvector of the largest eigenvalue of
\begin{align*}
\begin{pmatrix}
	\hat{D}_{k+lM/r,:}
\end{pmatrix}_{l=0}^{r-1}
\left(
\begin{pmatrix}
	\hat{D}_{k+lM/r,:}
\end{pmatrix}_{l=0}^{r-1}\right)^* \in\mathbb{C}^{r\times r}.
\end{align*}
Note that from our complexity analysis we concluded that $r=2$ or $r=3$ are the most suitable resampling rates and thus the above eigenvalue problem can be solved easily. Furthermore, note that $\hat{D}$ is the Fourier transform of real data and thus $\hat{D}_{k,:}=\overline{\hat{D}}_{M-k,:}$ for all $k=1,\ldots,M-1$. We can avoid solving for half of the eigenvectors by just substituting $\hat{\rho}_{M-k}=\overline{\hat{\rho}}_{k}$ for all $k=1,\ldots,\frac{M-1}{2}$. This way it is also guaranteed that $\rho\in\mathbb{R}^M$.

Next, we force the orthogonality requirement from Corollary \ref{samplingMatrix}, i.e., we want $R^TR=\mathbcal{I}_{M/r}$. Therefore, let $\delta_l=1$ for $l=0$ and $\delta_l=0$ for $l\neq0$. Then
\begin{align*}
	\delta_{l}&=\langle\rho,S_{lr}(\rho)\rangle=\langle\hat{\rho},F_M(S_{lr}(\rho))\rangle=\sum\limits_{k=0}^{M-1}|\hat{\rho}_k|^2e^{-\frac{2\pi i krl}{M}}\\
	&=\sum\limits_{k=0}^{M/r-1}\sum\limits_{j=0}^{j-1}|\hat{\rho}_{k+jM/r}|^2e^{-\frac{2\pi i (k+jM/r)rl}{M}}
	=\sum\limits_{k=0}^{M/r-1}|\alpha_k|^2e^{-\frac{2\pi ikl}{M/r}}
	=F_{M/r}^{-1}\begin{pmatrix}
		|\alpha_k|^2
	\end{pmatrix}_{k=0}^{M/r-1}
\end{align*}
This yields that $\alpha_k$ has constant magnitude with $|\alpha_k|=\sqrt[4]{\frac{r}{M}}$. This only leaves the phases of $\alpha_k$ to be chosen freely. However, the approximation error (\ref{rank1approximation}) stays the same independent from the chosen phases. Thus, we simply set $\alpha_k=\sqrt[4]{\frac{r}{M}}$. Our resampling matrix $R$ can then be constructed from the obtained vector $\rho$.

\begin{remark}
	For $r=2$ and $M$ divisible by $4$ the coefficients $\hat\rho_{M/4}$ and $\hat{\rho}_{3M/4}$ appear in the same eigenvalue problem (\ref{eigenvector_solution}). Hence the phase $\alpha_{M/4}$ is uniquely defined by the condition $\hat\rho_{M/4}=\overline{\hat{\rho}}_{3M/4}$. The same holds for $r=3$, $M$ divisible by $6$, and the phase $\alpha_{M/6}$. The other phases can be chosen freely which includes trivial ambiguities such as shifting the vector $\rho$ by $lr$ elements.
\end{remark} 

\begin{remark}
	The proposed strategy finds the optimal downsampling for given data $D$. Note that when artificially upsampling the data the error (\ref{scalingerror}) will be $0$ anyway. One may want to find the optimal upsampling operator by optimizing (\ref{griderror}). However, note that the coefficient $L(j,k,l)$ (\ref{Lgrid}) is smallest for $l$ close to $0$ or close to $M$, i.e., the error is small for data that is mostly low frequency. Thus any upsampling method that for example upsamples by setting high frequency coefficients to zero will perform well. The introduced Wavelet or Fourier resampling operators are designed in exactly this manner. Indeed, we have tried different upsampling techniques during our experiments with little to no difference in the results. Thus we spare the reader a more detailed analysis.
\end{remark}

\section{Numerics}

We compare the new fg-ORKA algorithm against the old ORKA approach in several tests. As resampling strategies we apply Wavelet, Fourier, and optimal resampling with $r=2$ and $r=3$ (no Wavelet resampling). In our first test, we compare the runtime complexity of all algorithms. The second test demonstrates the refined reconstruction of $\lambda$ for non-integer shifts. Afterwards, we analyze the stability under noise and for high frequency data. Last, we test the new algorithm on different applications.

\subsection{Runtime}

To compare the runtime of all variants, we apply the algorithm on randomly created data $D\in\mathbb{R}^{641\times 100}$, i.e., $100$ columns with $641$ pixels each, which is needed to allow sufficiently many downsampling steps when we test with large parameters $C$. The mean runtime over $50$ runs is measured for all experiments. In Figure \ref{fig:runtime_a} the runtime of all fg-ORKA variants is compared against the runtime of the original algorithm for $C=5$ and different parameters $K$. We can see that even for this small choice of $C$, the fg-ORKA variant outperforms the original approach by far (note that the y-axis uses a log-scale). Although both variants scale exponentially in $K$, the fg-ORKA approach grows much slower with a complexity of only $O(3^K)$ instead of $O(11^K)$.

Next, we want to compare the different downsampling approaches for varying parameter $C$. As seen in the first experiment, the original approach has a vastly larger runtime even for small $C$, for larger values of $C$ the algorithm might fail completely as it requires too much memory. For this reason, we only show the runtime of the different fg-ORKA variants in the next to experiments. Figures \ref{fig:runtime_b} and \ref{fig:runtime_c} show the runtime for all five variants with fixed parameter $K=5$ and $K=15$. For the first choice of $K$, we can clearly see that the runtime follows the complexity of the resampling technique. Here, Fourier resampling is the simplest and fastest approach, followed by Wavelet resampling, and the optimal downsampling strategy being the slowest. Furthermore, we can see that the larger resampling rate $r=3$ is usually faster than the smaller choice $r=2$. However, the difference between the strategy dimishes when increasing the paramter $K$, as can be seen in Figure \ref{fig:runtime_c}. If the paramter $K$ is large enough, the complexity of the ORKA algorithm dominates and the exact resampling strategy plays a very minor role for the runtime. We see that the runtime increases by steps of approximately the same size at certain thresholds. The thresholds are exactly the points $C$ at which the parameter $L$ is incremented by one due to the rounding performed on Equation (\ref{chooseL}). As $L$ increases by one, an additional iteration and thus an additional ORKA call has to be performed, which explains why the runtime increase is about the same for each step. We can clearly see, that the choice $r=3$ leads to a smaller number of iterations and thus to a faster runtime. However, especially for $C\leq16$ it is possible that the parameter $L$ is the same for both choices $r=2$ and $r=3$. In this case the runtime is indifferent.

\begin{figure}
	\begin{center}
	\subfloat[\label{fig:runtime_a}]{\includegraphics[width=0.3\textwidth]{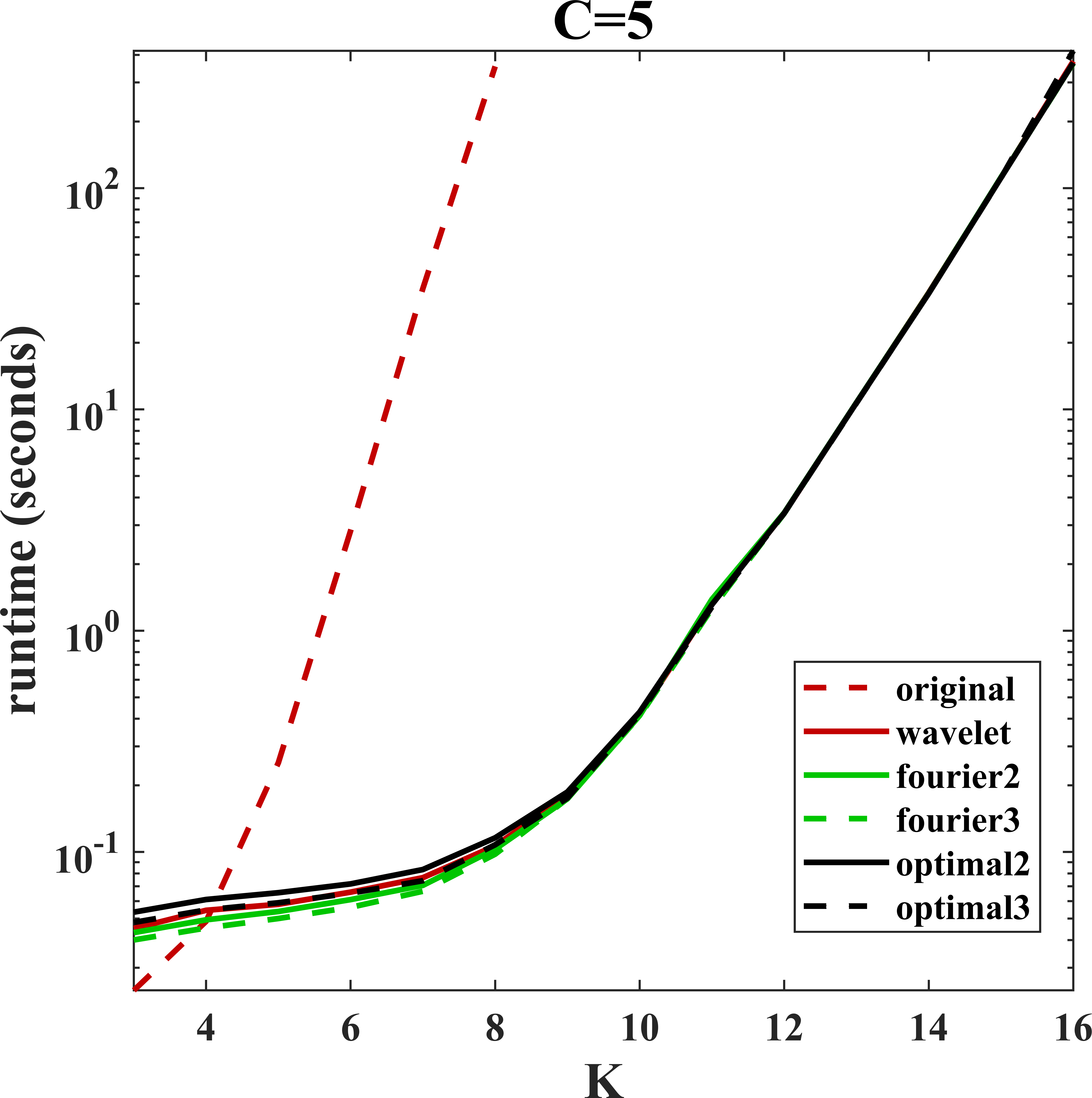}}\ 
	\subfloat[\label{fig:runtime_b}]{\includegraphics[width=0.3\textwidth]{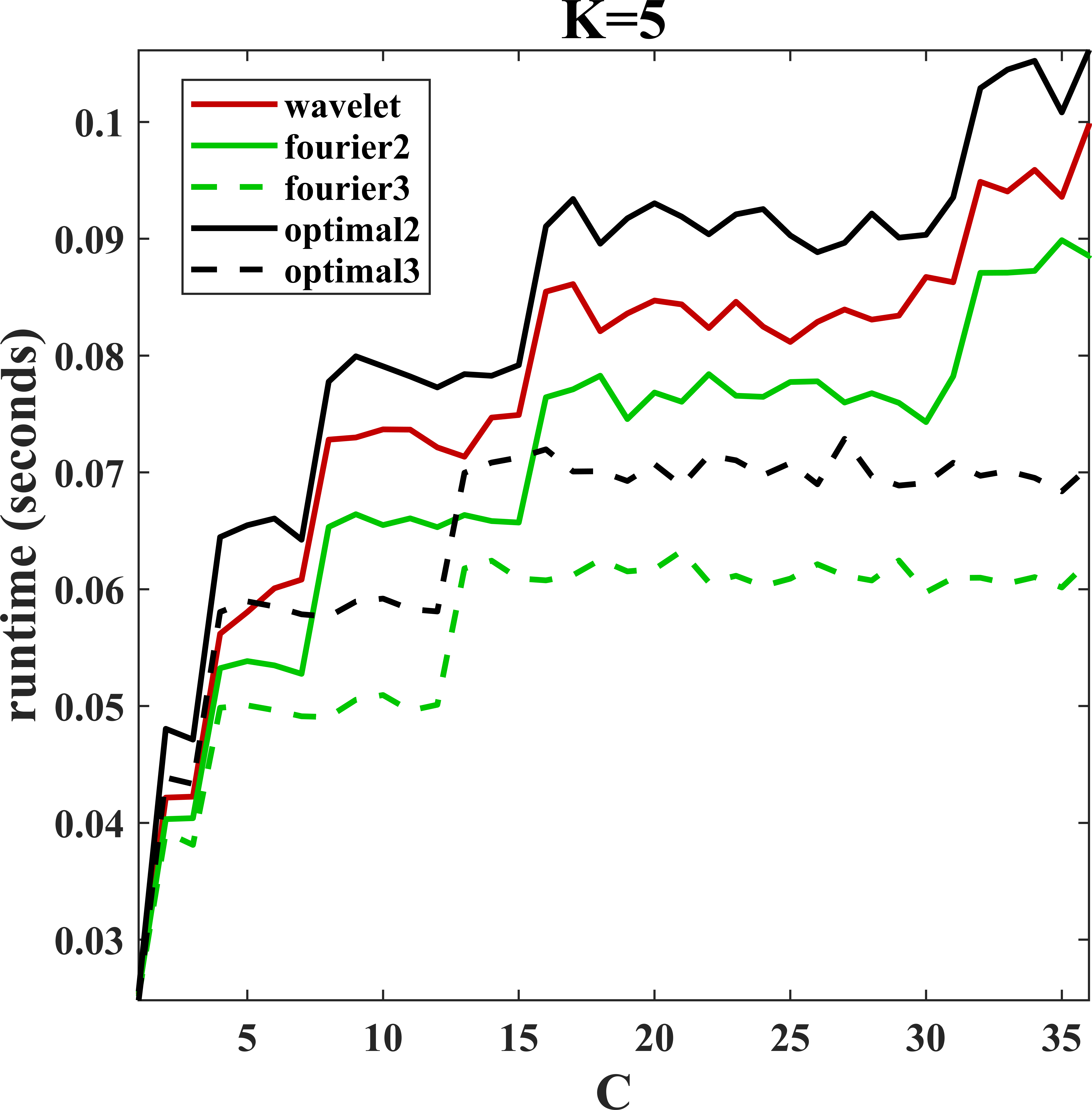}}\ 
	\subfloat[\label{fig:runtime_c}]{\includegraphics[width=0.3\textwidth]{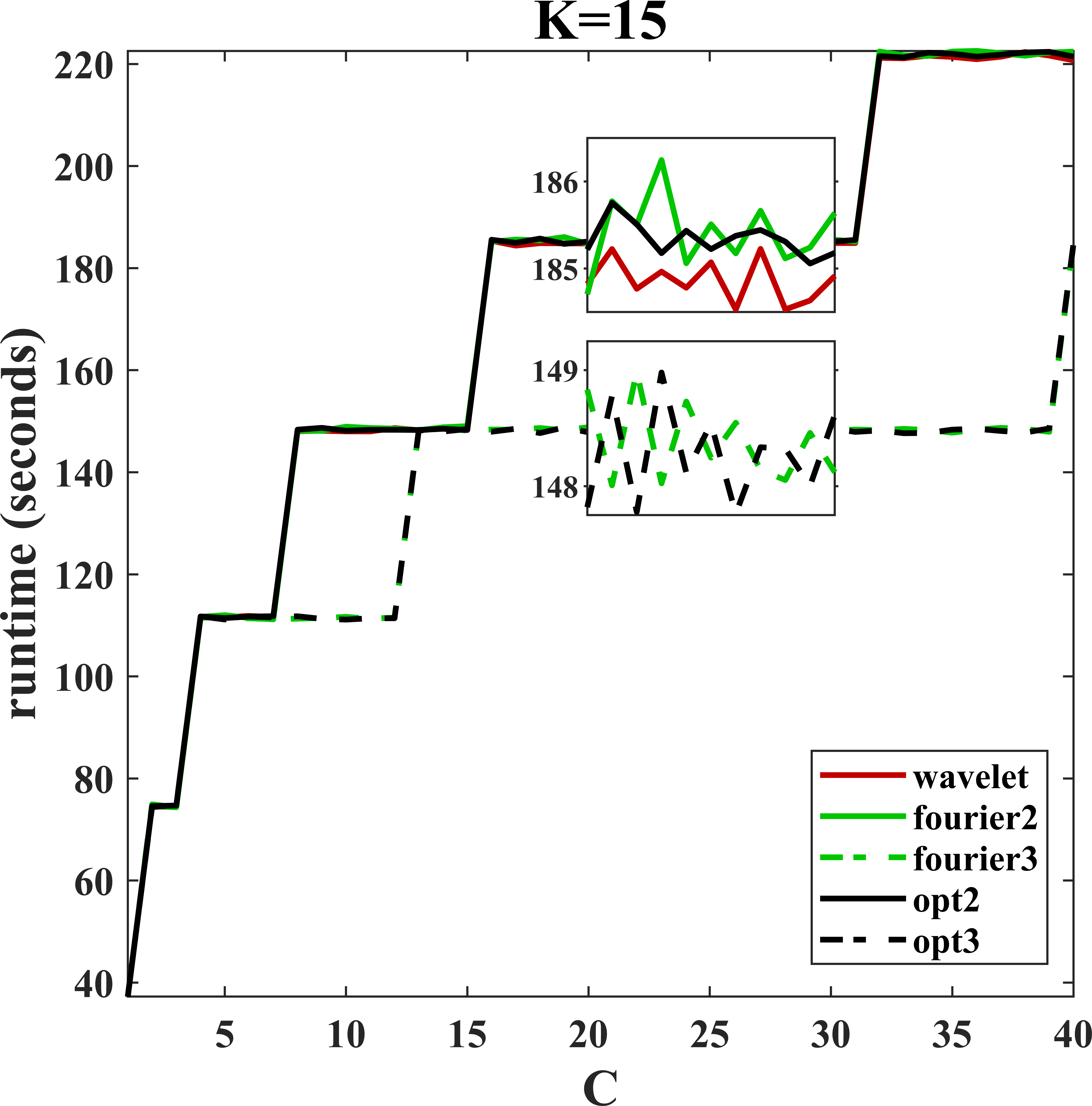}}
	\caption{Runtime of fg-ORKA in different settings: a) $C=5$ for increasing $K$ compared to the original ORKA approach; b) $K=5$ for increasing $C$; c) $K=15$ for increasing $C$.}
	\label{fig:runtime}
	\end{center}
\end{figure}

\subsection{Non-integer shift}

In our next experiment we test the reconstruction of non-integer shift vectors $\lambda$ by the proposed upsampling strategy. Therefore, we create random test data $\tilde{D}\in\mathbb{R}^{500\times100}$ in the following way. First, we create a random shift vector $\tilde{\lambda}\in\mathbb{Z}^{100}$ where $\tilde{\lambda}_k-\tilde{\lambda}_{k+1}\in\{-4,-3,\ldots,4\}$ is uniformly distributed. Second, we create a random sampling vector $d\in\mathbb{R}^{500}$ sampled from a standard Gauss distribution. Last, we set $\tilde{D}_{:k}=S_{k}(g*d)$ where $g\in\mathbb{R}^{500}$ is defined as
\begin{align*}
	g_k=\begin{cases}
		e^{-(0.4(k-6))^2} & k=1,\ldots,11 \\
		0 & \text{otherwise}
	\end{cases}.
\end{align*}
This means, the matrix $\tilde{D}$ exactly fits the proposed object model. However, we now downsample by a factor of $5$ and define the data $D\in\mathbb{R}^{100\times100}$ with $D_{j,k}=\tilde{D}_{5j,k}$. Now, $D$ requires a non-integer shift vector of $\lambda=\frac{\tilde{\lambda}}{5}$. Note that the convolution with a kernel such as $g$ is required as otherwise the columns can be completely independent random samples after the downsampling step.

We use the fg-ORKA algorithm to reconstruct the shifts with different levels of upsampling $J$. We measure the error between the original shift $\lambda^{\text{org}}$ and the reconstructed shift $\lambda^{\text{rec}}$ as
\begin{align}\label{lambda_error_measure}
	\text{Error}(\lambda^{\text{rec}})=N^{-1}\left\|\lambda^{\text{org}}-\lambda^{\text{rec}}-\text{mean}\left(\lambda^{\text{org}}-\lambda^{\text{rec}}\right)\right\|_1,
\end{align} 
where $N$ is the number of columns in $D$, i.e., $N=100$ for this experiment. The mean value is subtracted as the solution of the ORKA problem (\ref{ORKAopt}) is ambiguous. For any shift $\lambda$ the shifts $\lambda+n$ with $n\in\mathbb{Z}$ gives the same minimum value. Thus, we say that the reconstruction is optimal if $\lambda^{\text{rec}}$ fits $\lambda^{\text{org}}$ upto a constant shift.

Figure \ref{fig:nonintshift} shows the mean reconstruction error over $50$ runs for the three different upsampling methods: Wavelet, Fourier ($r=2$), and Fourier ($r=3$). (Note that the y-axis uses a logarithmic scale.) We see that the reconstruction gets more accurate with increasing levels of upsampling $J$. Nevertheless, a small choice of $J$ seems sufficient as the error does no longer improve much after the first $4$ to $5$ iterations. As expected, the approximation error is lower for a larger parameter choice $K=15$. Furthermore, we note that Fourier upsampling with $r=3$ yields the best results in the first iterations, since upsampling by a factor for $3$ gives a higher resolution compared to $r=2$. The final results, however, is worse for $r=3$. We assume that this is because the setting $r=3$ is more prone to errors (see Remark \ref{rmk:numbersBaseRelation}). For us the most suprising result of this experiment is, that upsampling using the simple Fourier approach returns better results compared to a Wavelet based upsampling.

For $J=0$ all methods are equivalent to the original ORKA method. In this case we expect the optimal path to be $\text{round}(\lambda^{\text{org}})$. Since $\lambda^{\text{org}}$ is drawn from a uniform distribution we can calculate the expected error in this case as
\begin{align*}
	\frac{1}{9}\sum\limits_{k=-4}^4\left|\frac{k}{5}-\text{round}\left(\frac{k}{5}\right)\right|=\frac{4}{15}\approx0.2667
\end{align*}
which is about the value that we achieve with $K=15$.

\begin{figure}
	\begin{center}
		\includegraphics[width=0.3\textwidth]{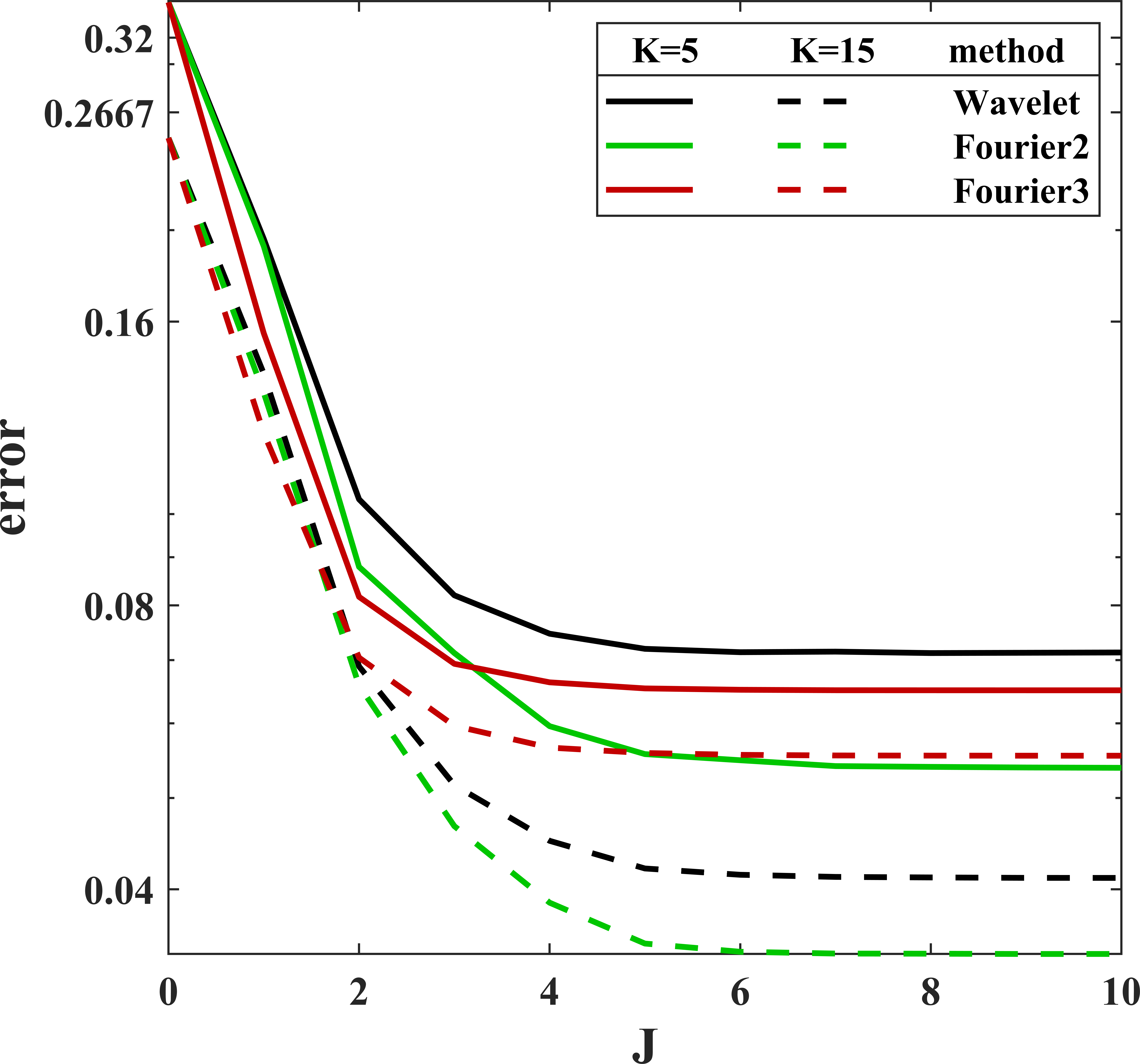}
		\caption{Approximation error for non-integer shifts and different levels of upsampling.}
		\label{fig:nonintshift}
	\end{center}
\end{figure}

\subsection{Noisy and high frequency data}

In this subsection, we test the stability of all downsampling approaches. For this, we add different levels of noise to the data and also apply a high-pass filter that sets the lower frequencies to $0$. As seen in the error analysis, we expect the approximation error to get worse when the data only contains high frequencies. Moreover, the suggested Fourier and Wavelet downsampling are low-pass filters and thus should also struggle with this kind of data.

In our first test, we take the mean approximation error (\ref{lambda_error_measure}) for $K=5$ over $200$ runs with random data $D\in\mathbb{R}^{641\times100}$ constructed as in the previous experiment (except for the downsampling step). We use data with a random shift vector $\lambda\in\mathbb{Z}^{100}$ that has a maximum shift difference $|\lambda_k-\lambda_{k+1}|\leq C$ of $C=5$ and $C=27$. For $C=5$ we can also calculate the mean approximation error of the original ORKA method, for $C=27$ the calculation fails due to insufficient memory. The results can be seen in Figure \ref{fig:K5noisePlots_allFreq_slope5} and \ref{fig:K5noisePlots_allFreq_slope27}. For $C=5$ all downsampling methods with $r=2$ perform equally and are as good as the original ORKA method. They are able to reconstruct the original shift even for noisy data up to a PSNR of about $15$. The downsampling methods with $r=3$ preform slightly worse. For $C=27$ (Figure \ref{fig:K5noisePlots_allFreq_slope27}) the approximation errors are higher on average as we have more possible paths $\lambda$ in this case. Here, the optimal downsampling methods both perform worse than the Wavelet or Fourier based methods. This shows that although the loss of information in each downsampling step is minimized, this is not necessarily the best way to preserve the information about the shift. Interestingly, for $C=27$ the Fourier downsampling technique with $r=3$ performs best. Here, it pays off that $r=3$ requires less downsampling steps and thus less iterations.

In Figure \ref{fig:K5noisePlots_highFreq_slope5} and \ref{fig:K5noisePlots_highFreq_slope27} we repeat the same experiments but now apply a high-pass filter to the randomly created data beforehand. The filter removes the $160$ lowest frequencies (out of $641$). We observe that for both cases $C=5$ and $C=27$ the optimal downsampling strategies now perform better compared to the other strategies with same resampling factor $r$. Only the original ORKA method is able to construct the correct shift up to a PSNR of about $8$. Wavelet downsampling can handle the high frequency data slightly better than the more simple Fourier based approach. Especially for $C=27$ the Fourier downsampling with $r=3$, which performed best before, now suffers from the missing low frequency information and performs worst.

\begin{figure}
	\begin{center}
		\subfloat[\label{fig:K5noisePlots_allFreq_slope5}]{\includegraphics[width=0.3\textwidth]{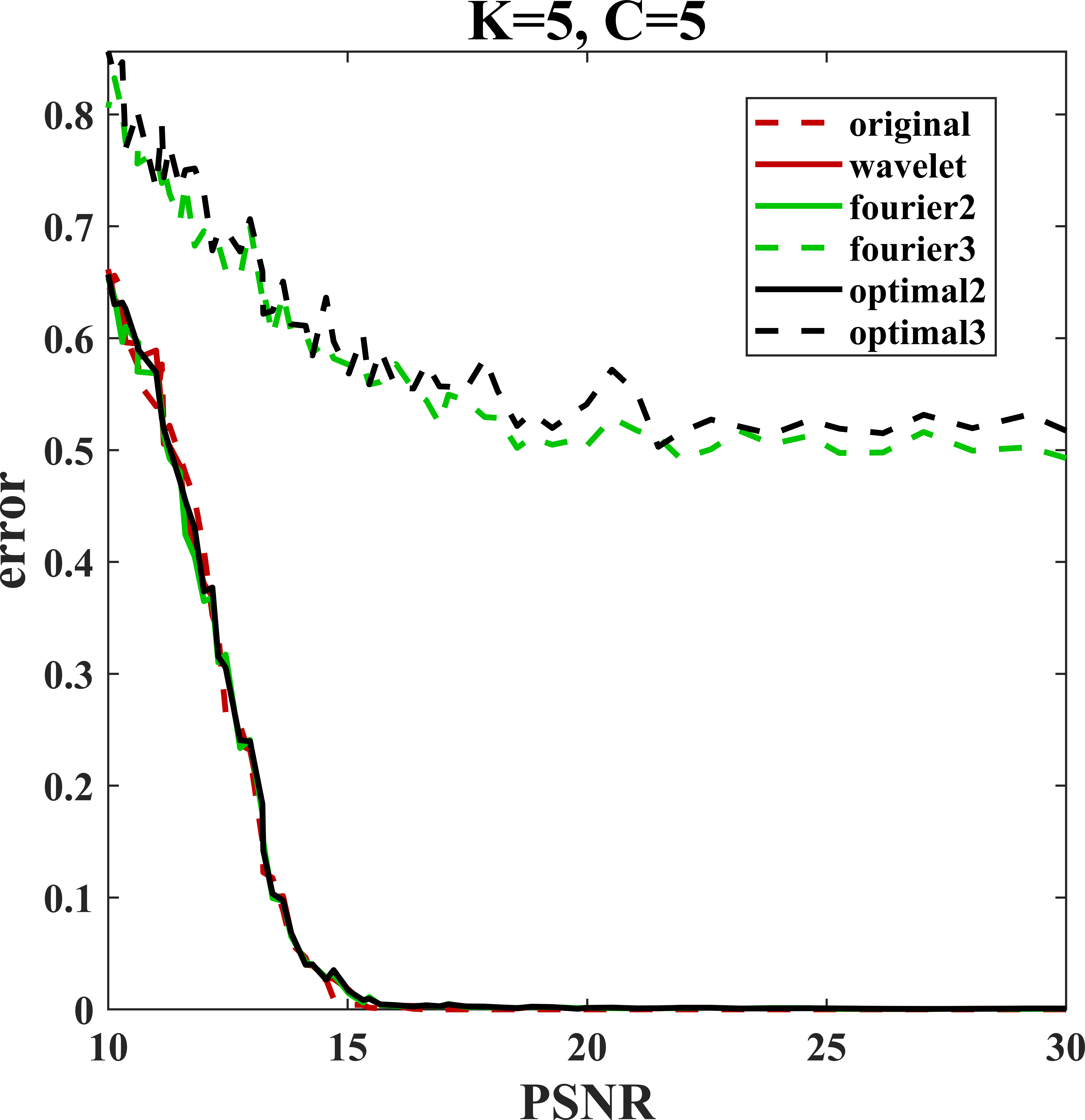}}\ 
		\subfloat[\label{fig:K5noisePlots_allFreq_slope27}]{\includegraphics[width=0.3\textwidth]{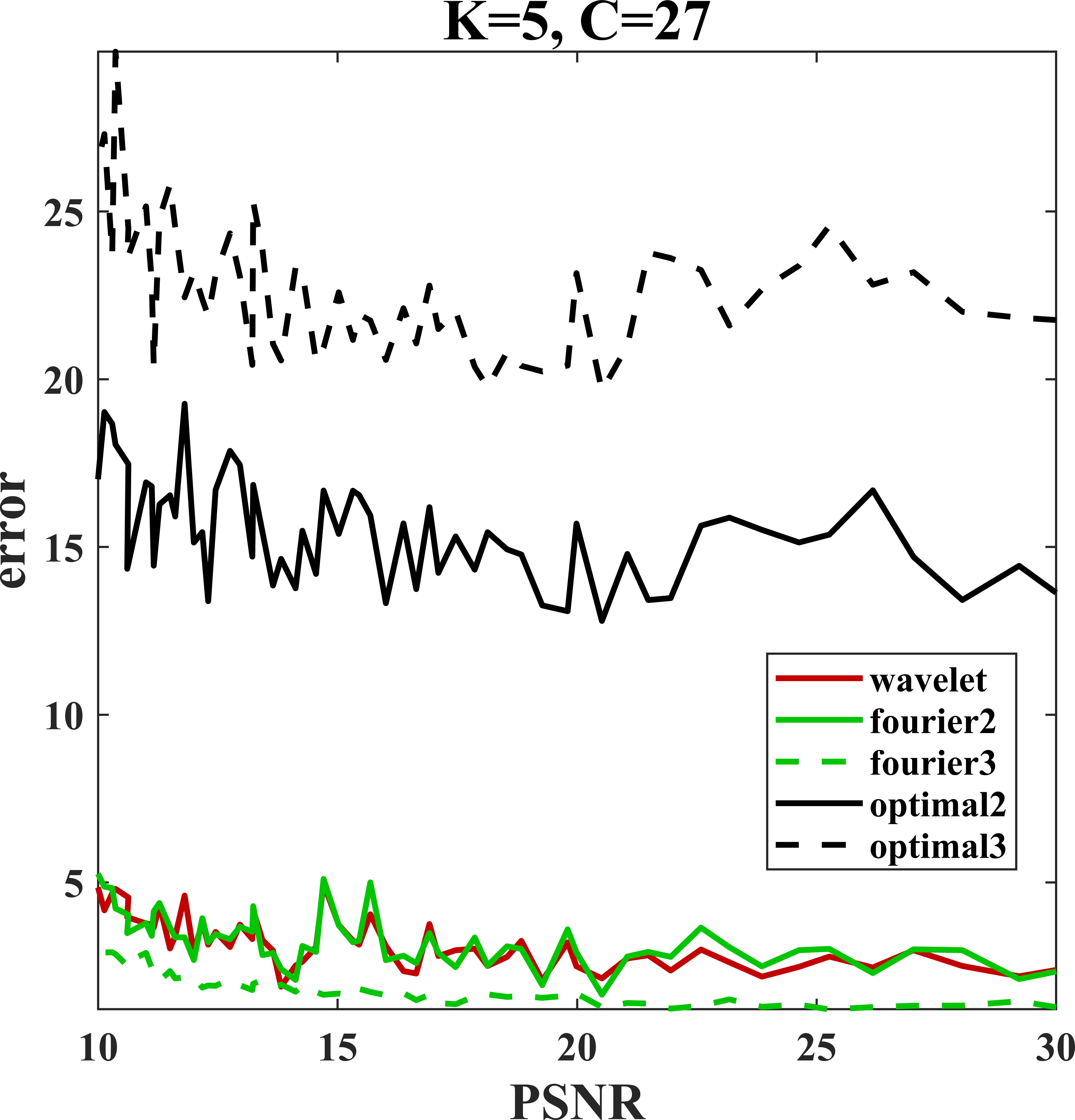}}\\
		\subfloat[\label{fig:K5noisePlots_highFreq_slope5}]{\includegraphics[width=0.3\textwidth]{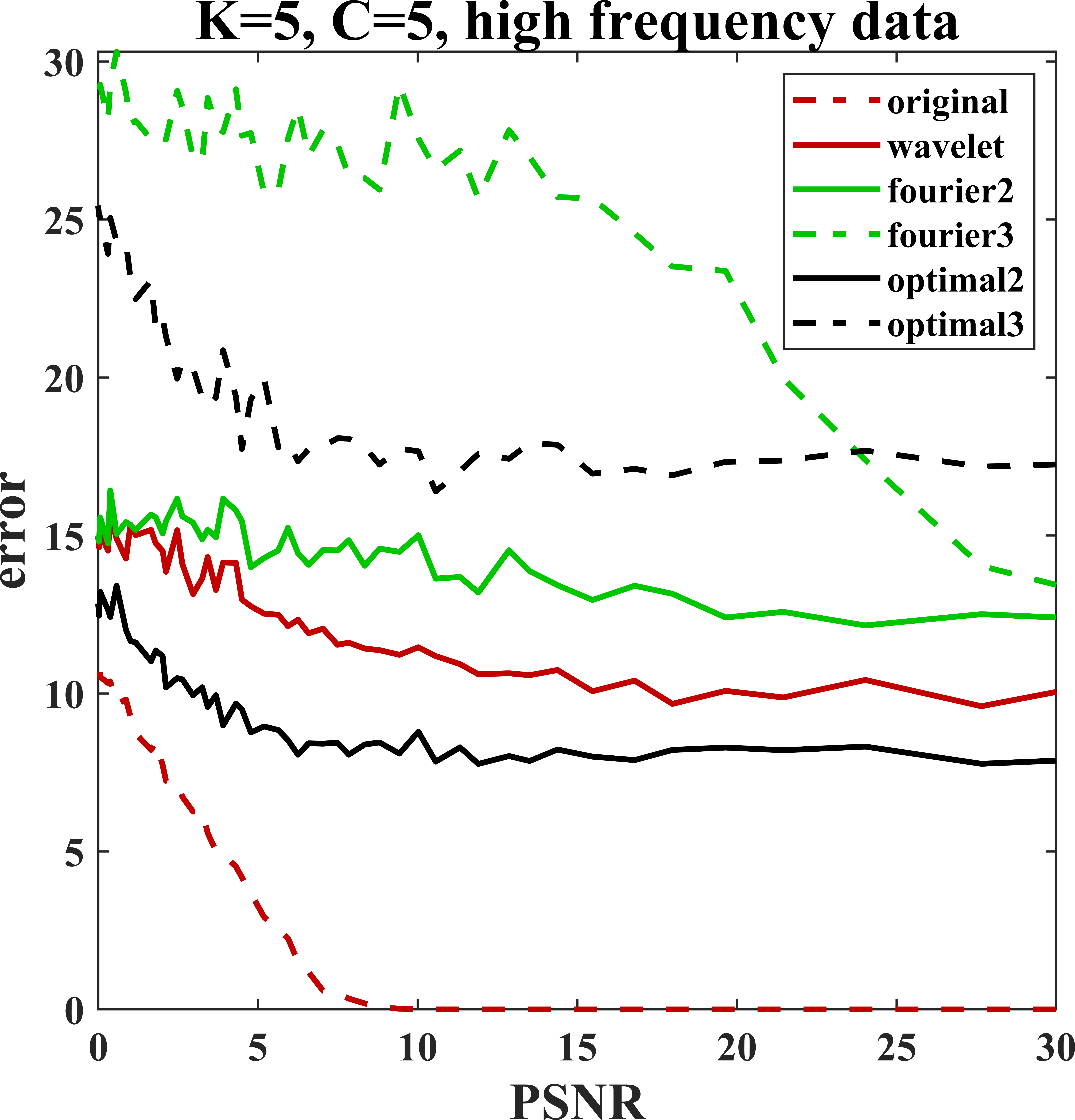}}\ 
		\subfloat[\label{fig:K5noisePlots_highFreq_slope27}]{\includegraphics[width=0.3\textwidth]{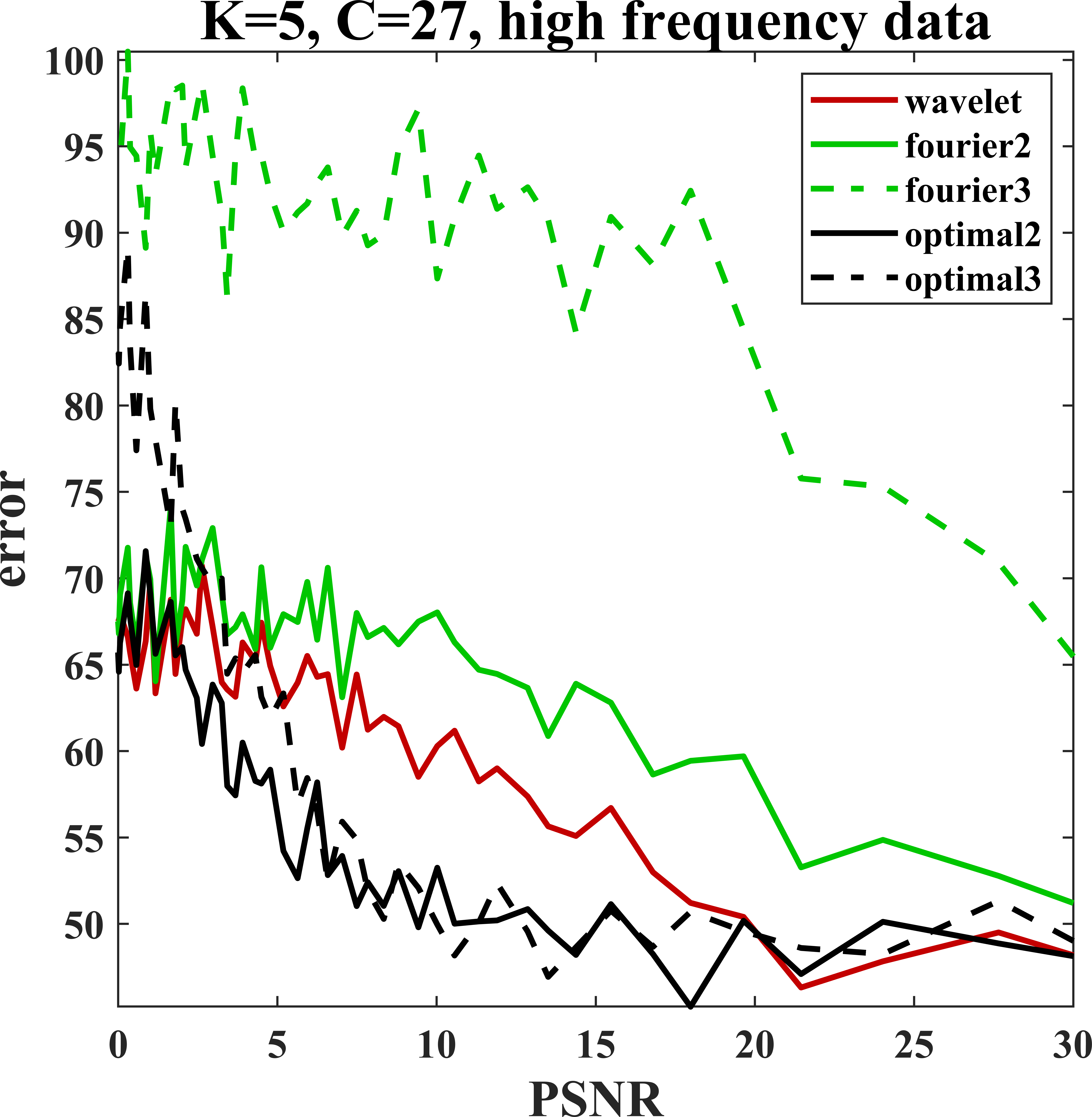}}
		\caption{Mean approximation error for $K=5$ and noisy data with a maximum shift difference of $C=5$ (left) or $C=27$ (right). In (c) and (d) only high frequency data was used, where the $160$ lowest frequencies are $0$.}
		\label{fig:K5noisePlots}
	\end{center}
\end{figure}

We run the above experiment for further combinations of noisy and high-pass filtered data to evaluate which downsampling method performs best in these cases. Figure \ref{fig:K5bestMethod} shows the best method for all combinations. Here the x-axis shows the PSNR value and the y-axis gives the number of filtered low frequencies. Note that we gradually increased the Gaussian noise added to the data, but the PSNR increases faster the more low frequencies are filtered. This is why Figure \ref{fig:K5bestMethod} shows a curved image. For $C=5$ we observe a similar pattern as expected from our observations before. As long as most of the low frequencies are preserved, Fourier and Wavelet based downsampling with $r=2$ dominates the image. For high-pass filtered data however the best method is by far the optimal downsampling approach. For $C=27$ we get a similar result where the optimal downsampling again being the best method for most of the high frequency data. However, since $r=3$ requires less iterations in this case, we also see Fourier downsampling with $r=3$ and optimal downsampling with $r=3$ show up in some areas. While optimal downsampling with $r=3$ can be used on some high-frequency data, Fourier downsampling with $r=3$ is best for data that still contains low frequency information. It is then overtaken by Wavelet based downsampling. Fourier downsampling with $r=2$ only appears in a small area where the PSNR is not too low and not too many low frequencies are filtered out. Last, we want to point out the small lengthy region in both images that appear around $150$ filtered low frequencies and starts at the right boarder. As of now we are not sure why this phenomenon occurs.

\begin{figure}
	\begin{center}
		\subfloat[\label{fig:K5bestMethod_slope5}]{\includegraphics[width=0.3\textwidth]{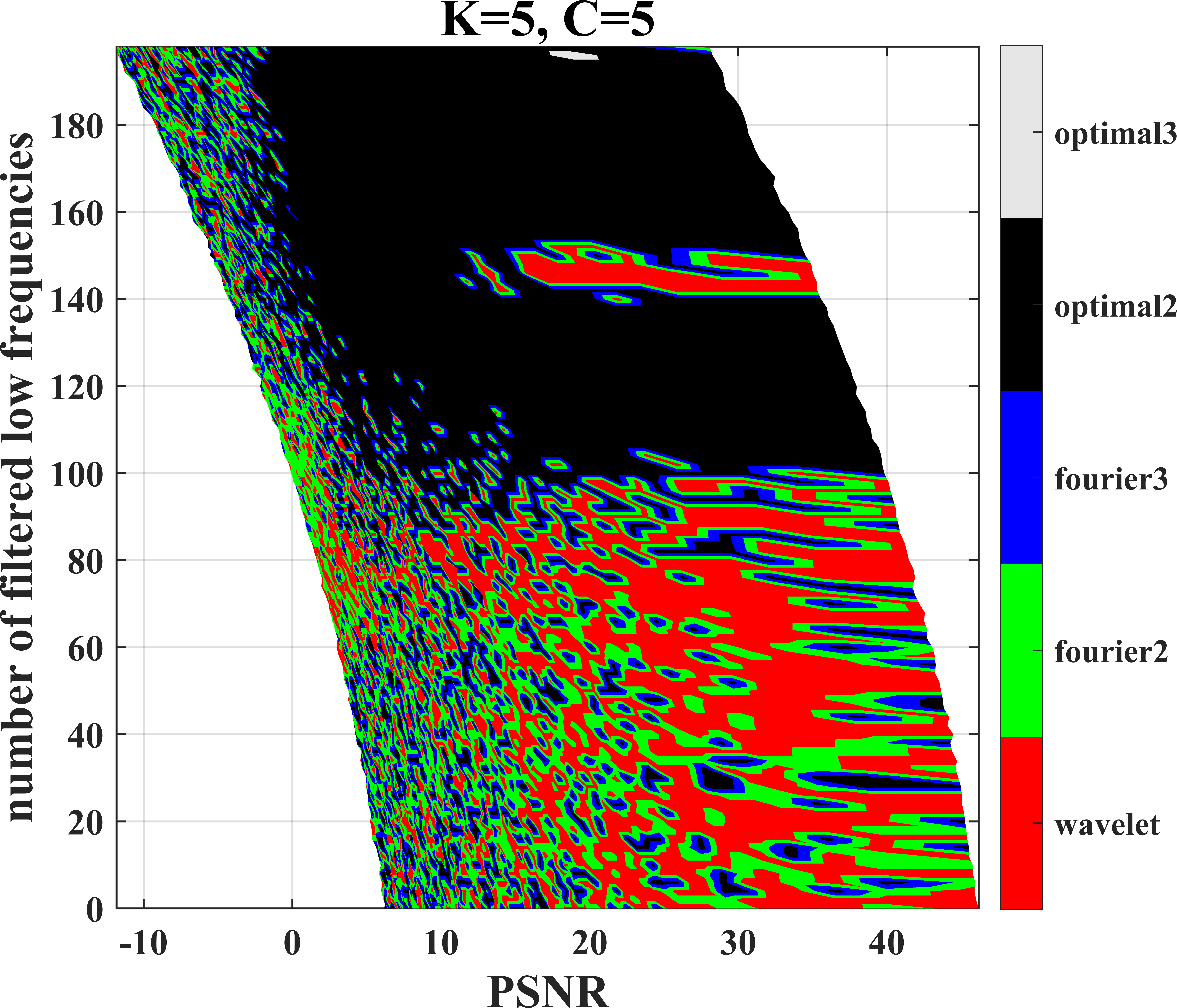}}\ 
		\subfloat[\label{fig:K5bestMethod_slope27}]{\includegraphics[width=0.3\textwidth]{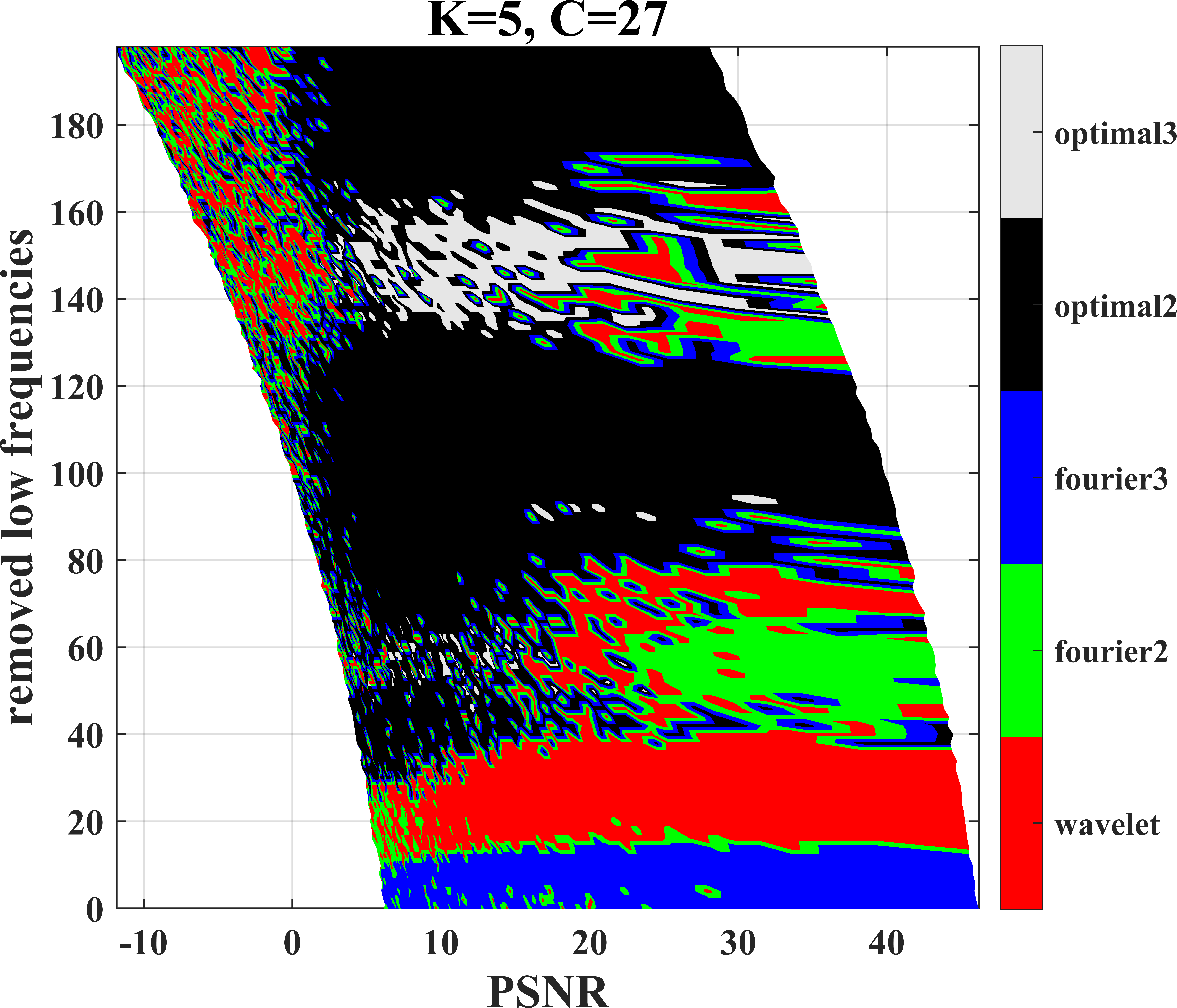}}
		\caption{Best downsampling methods on average for $K=5$, different noise levels and data with low frequencies removed. The maximum shift difference is $C=5$ (a) and $C=27$.}
		\label{fig:K5bestMethod}
	\end{center}
\end{figure}

We repeat the same experiment now with a much higher parameter $K=15$. Figure \ref{fig:K15noisePlots} shows the obtained approximation errors for $C=5$, $C=27$, random data, and random high frequency data (compare Figure \ref{fig:K5noisePlots}). For $C=5$ (Figure \ref{fig:K15noisePlots_allFreq_slope5} and \ref{fig:K15noisePlots_highFreq_slope5}) we also show the results of the original ORKA method with parameter $K=5$ again, for $K=15$ the original method will fail due to insufficient memory. We can see that the approximation errors are smaller compared to $K=5$. The relation between the different upsampling methods is similar to the previous case. Resamplings with $r=2$ performs better on average than $r=3$, for the high-pass filtered data the optimal downsampling approach returns the best results, followed by Wavelet downsampling, and last Fourier downsampling. With the higher choice of $K=15$ we can also outperform the original ORKA for highly noised data. This shows that even in a setup where the original algorithm can be used it can be beneficial to switch to the iterative version. Moreover, for $C=27$ we have again the downsampling method with $r=3$ perform good. Here, the optimal approach is best for high frequency data, while Fourier based downsampling with $r=3$ performs good unfiltered random data.

\begin{figure}
	\begin{center}
		\subfloat[\label{fig:K15noisePlots_allFreq_slope5}]{\includegraphics[width=0.3\textwidth]{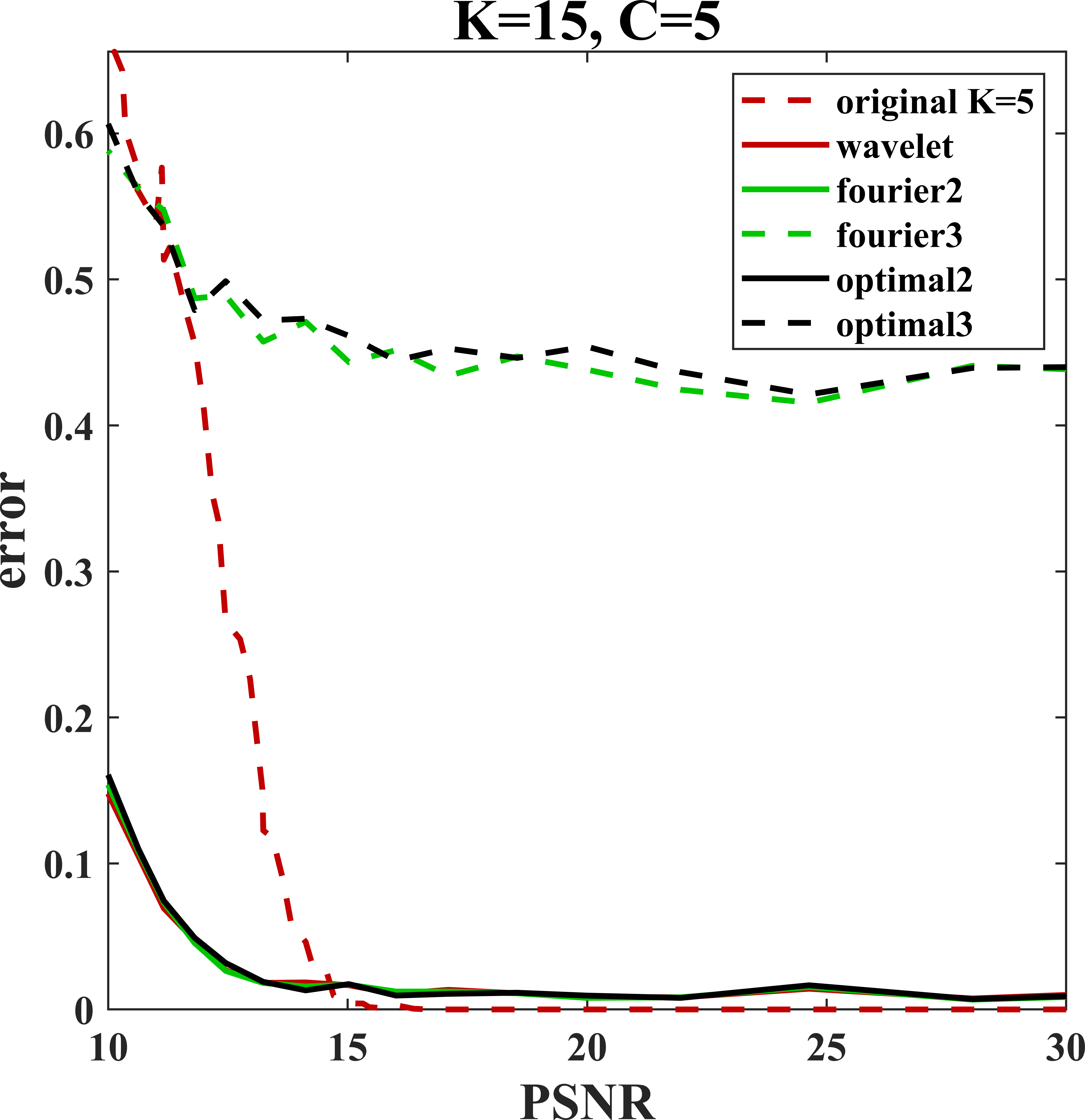}}\ 
		\subfloat[\label{fig:K15noisePlots_allFreq_slope27}]{\includegraphics[width=0.3\textwidth]{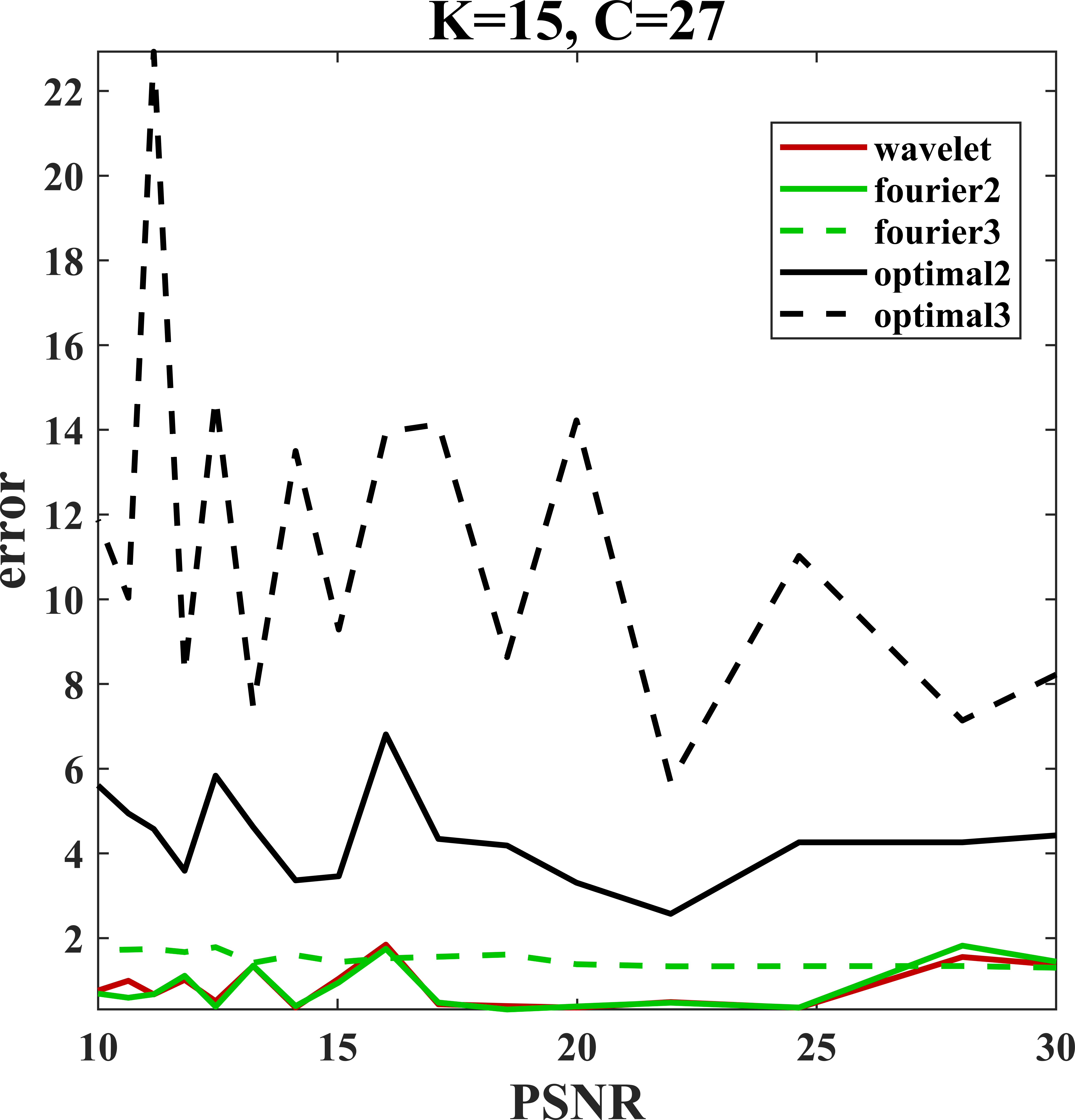}}\\
		\subfloat[\label{fig:K15noisePlots_highFreq_slope5}]{\includegraphics[width=0.3\textwidth]{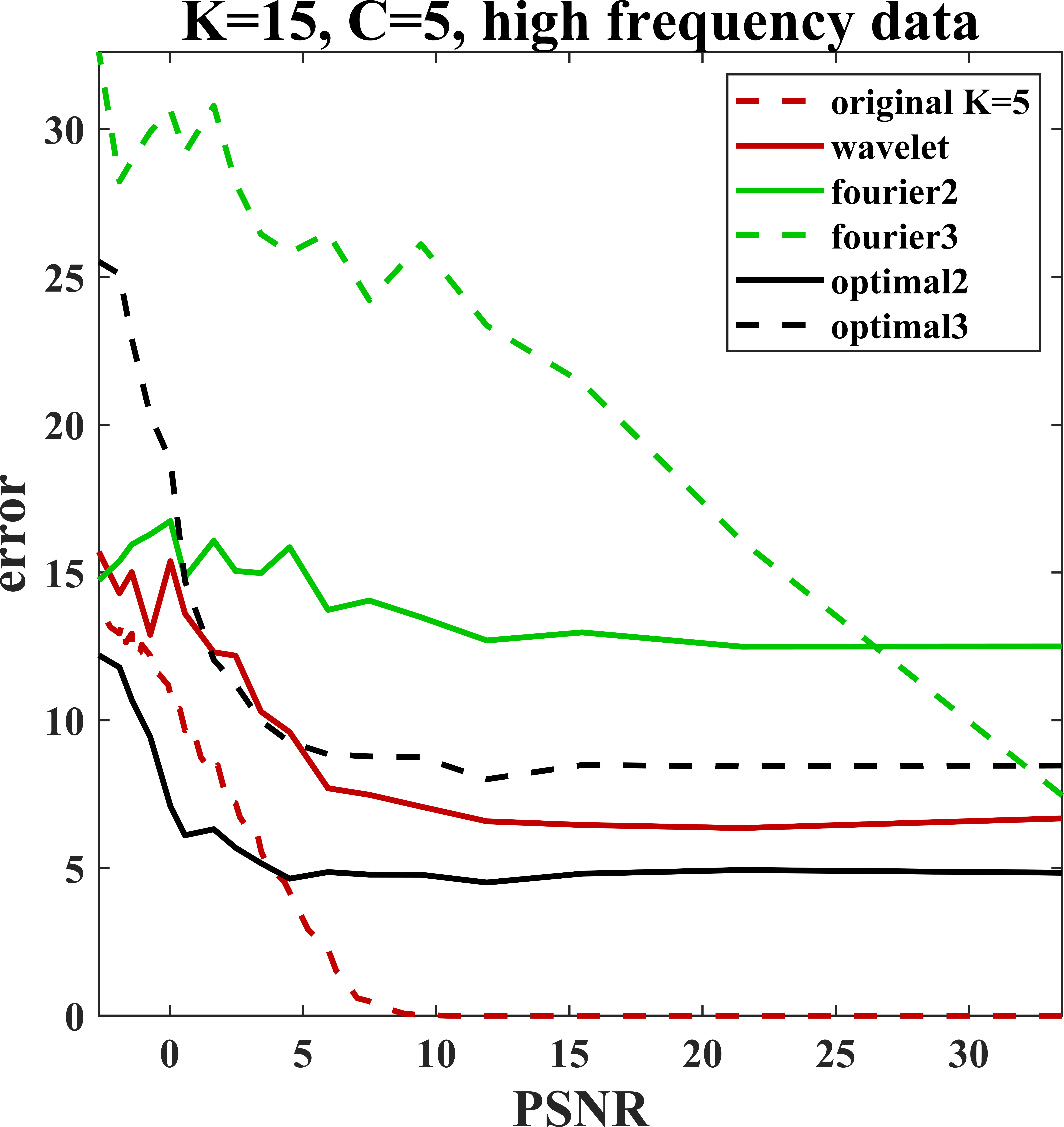}}\ 
		\subfloat[\label{fig:K15noisePlots_highFreq_slope27}]{\includegraphics[width=0.3\textwidth]{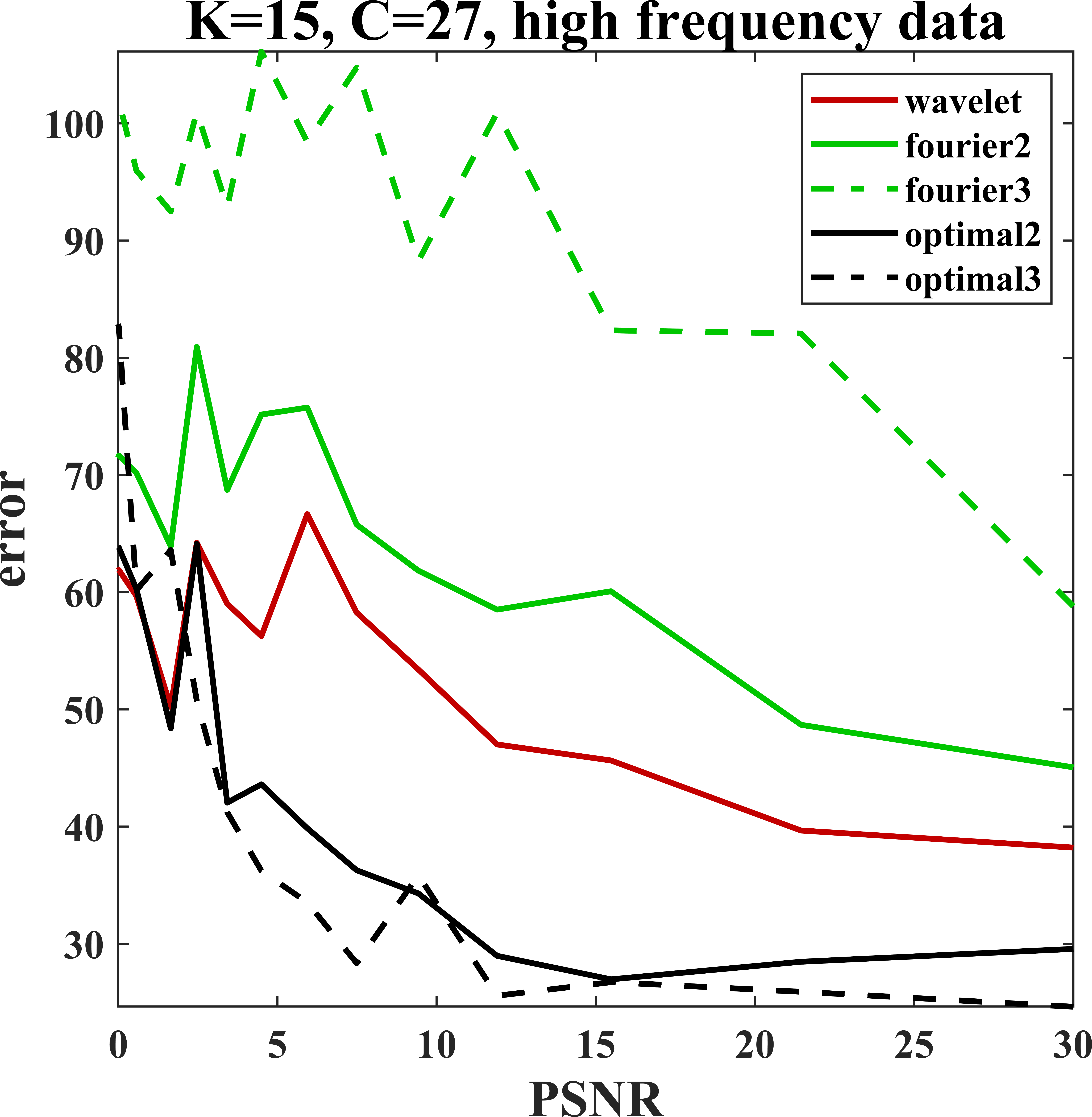}}
		\caption{Mean approximation error for $K=15$ and noisy data with a maximum shift difference of $C=5$ (left) or $C=27$ (right). In (c) and (d) only high frequency data was used, where the $160$ lowest frequencies are $0$.}
		\label{fig:K15noisePlots}
	\end{center}
\end{figure}

\subsection{Application data}

In our last two experiments we demonstrate fg-ORKA on data from two different applications. First, track seismic waves in geophysical data. Here, the upsampling strategy is used to get a finer result as with the original ORKA approach. In the second experiment, we try fg-ORKA on a soccer video. The video shows many fast moving objects of different sizes. Furthermore, the camera is not fixed but moving throughout the scene. Tracking single players or other objects of interest within this video is a hard task and we will use this example to demonstrate the current limitations of our technique.

The real geophysical data shown in Figure \ref{fig:seismic_org} shows a large seismic wave a the top of the image. We use the original ORKA algorithm and the new fg-ORKA to track this wave. The parameters used are $C=5$ and $\mu=100$. Following the results from the previous tests we use $J=5$ levels of upsampling and the Fourier based upsampling technique with $r=2$. For the original ORKA algorithm we set $K=8$, which is the largest value possible on the used machine, and for fg-ORKA we set $K=15$. With this setup both methods take about $10$ minutes to complete. The reconstructed movement $\lambda$ of the seismic wave is shown in Figure \ref{fig:seismic_ORKA} (original ORKA) and \ref{fig:seismic_fgORKA} (fg-ORKA). The black lines indicate the reconstructed movement, the original data is drawn in the background as a comparison. We can clearly see that fg-ORKA fits much better to the seismic structure and thus reconstructs a more accurate movement.
 
\begin{figure}
	\begin{center}
		\subfloat[\label{fig:seismic_org}]{\includegraphics[width=0.3\textwidth]{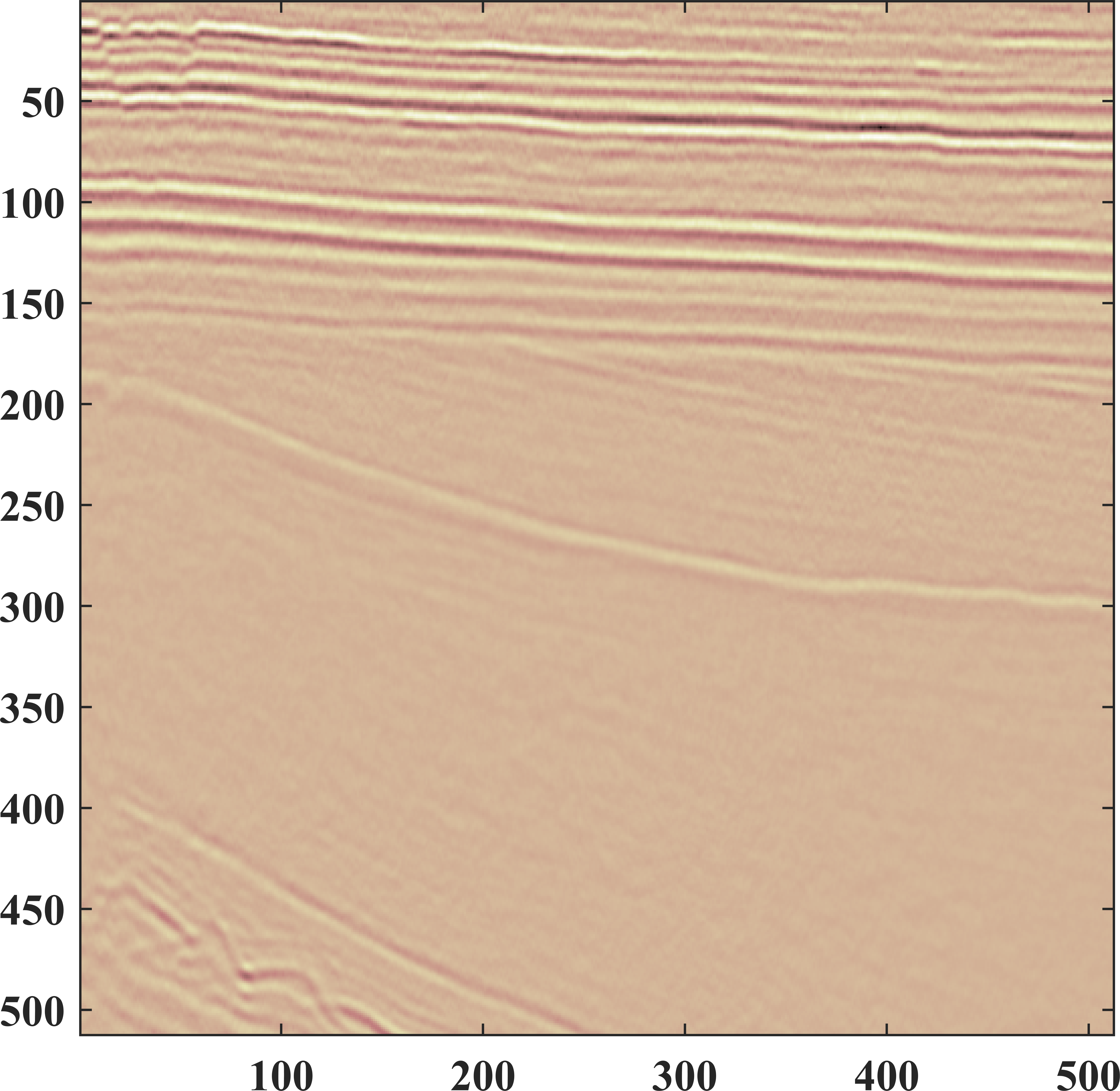}}\ 
		\subfloat[\label{fig:seismic_ORKA}]{\includegraphics[width=0.3\textwidth]{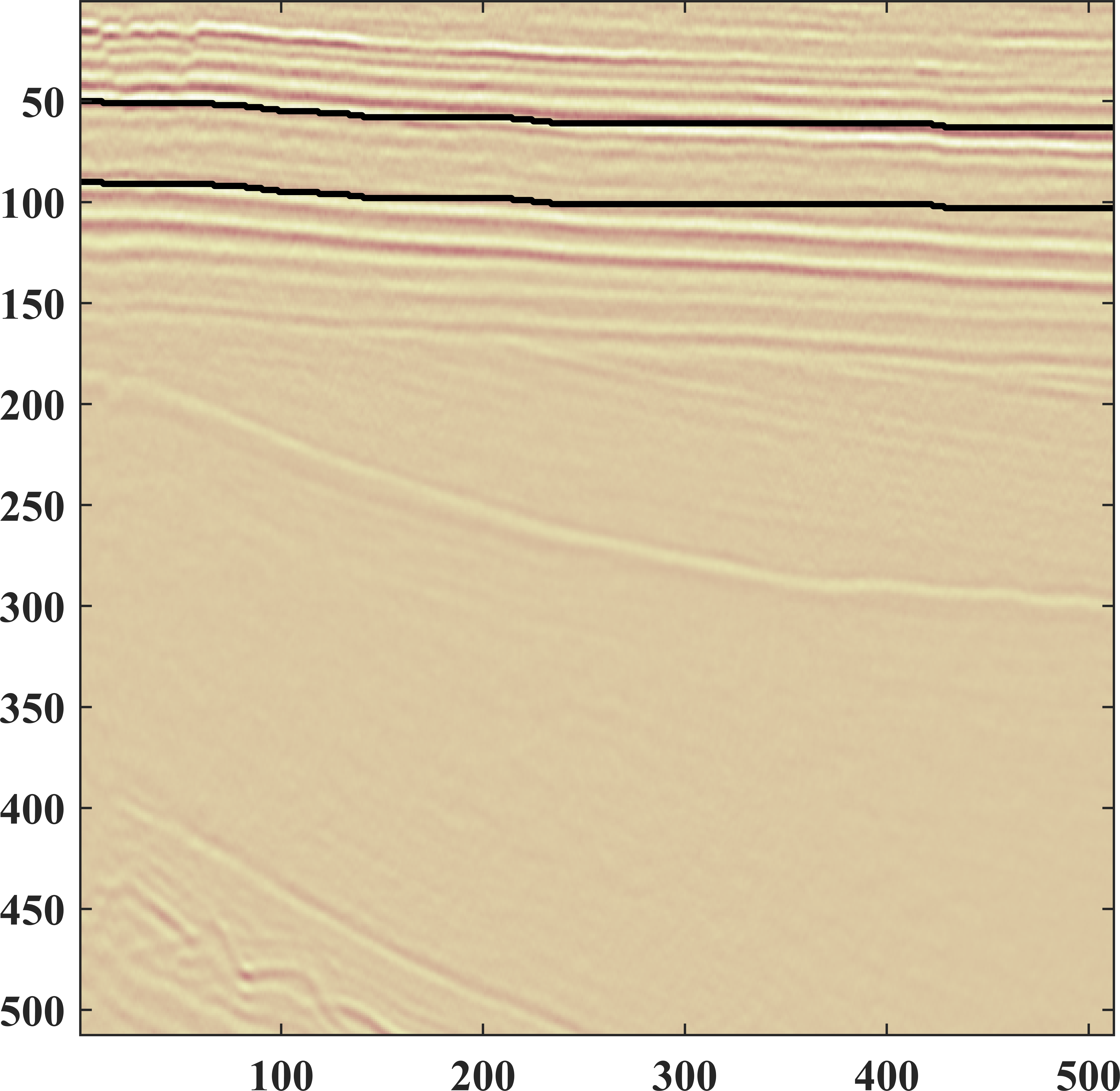}}\ 
		\subfloat[\label{fig:seismic_fgORKA}]{\includegraphics[width=0.3\textwidth]{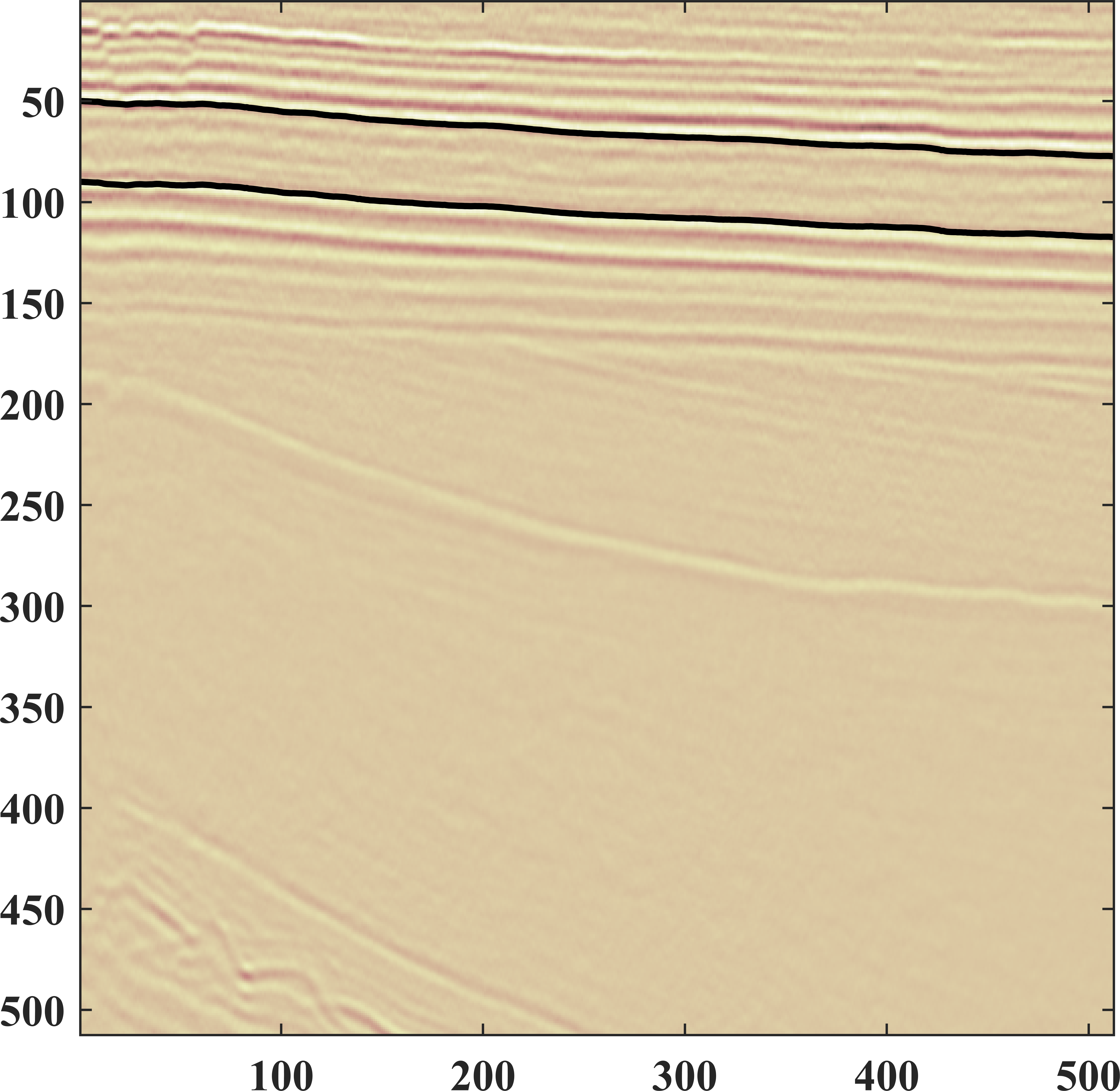}}
		\caption{Tracking a seismic event in the post-stack data (a) using the original ORKA (b) compared to fg-ORKA with $J=5$ upsampling steps (c).}
		\label{fig:seismic}
	\end{center}
\end{figure}

In our last experiment we compare the original ORKA algorithm to fg-ORKA on a soccer video from the "UCF Sports Action Data Set" from the UCFCenter of Research in Computer Vision \cite{Rodriguez08,Soomro14}. The video has a frame size of $576\times720$ pixels and has a total of $65$ frames. Figure \ref{fig:soccer} shows the first and last frame of the video to give the reader an impression of the scene: A recording of a soccer match showing multiple players, a referee and the ball. Throughout the scene the camera rotates in the right direction which especially changes the advertisements shown in the background. The video poses many challenges for our algorithm. First, the players and camera are moving with a fast speed which requires a large choice of $C$. Second, the players and especially the ball are small objects compared to the video size what makes them hard to track. Furthermore, in our algorithm we consider periodic shifts, i.e., any data that is shifted to the right/bottom out of the frame will appear again at the left/top of the frame. This is of cause not the case for videos but rather a restriction of our model so far. Another difference to our model is, that overlapping objects in the video do not add up their gray scale values, but rather the object in front covers the object in the background. Lastly, because we are dealing with multi-dimensional data now, the complexity of ORKA scales as $O((2C+1)^{2K})$ and $O(3^{2K})$ for fg-ORKA.

\begin{figure}
	\begin{center}
		\includegraphics[width=0.49\textwidth]{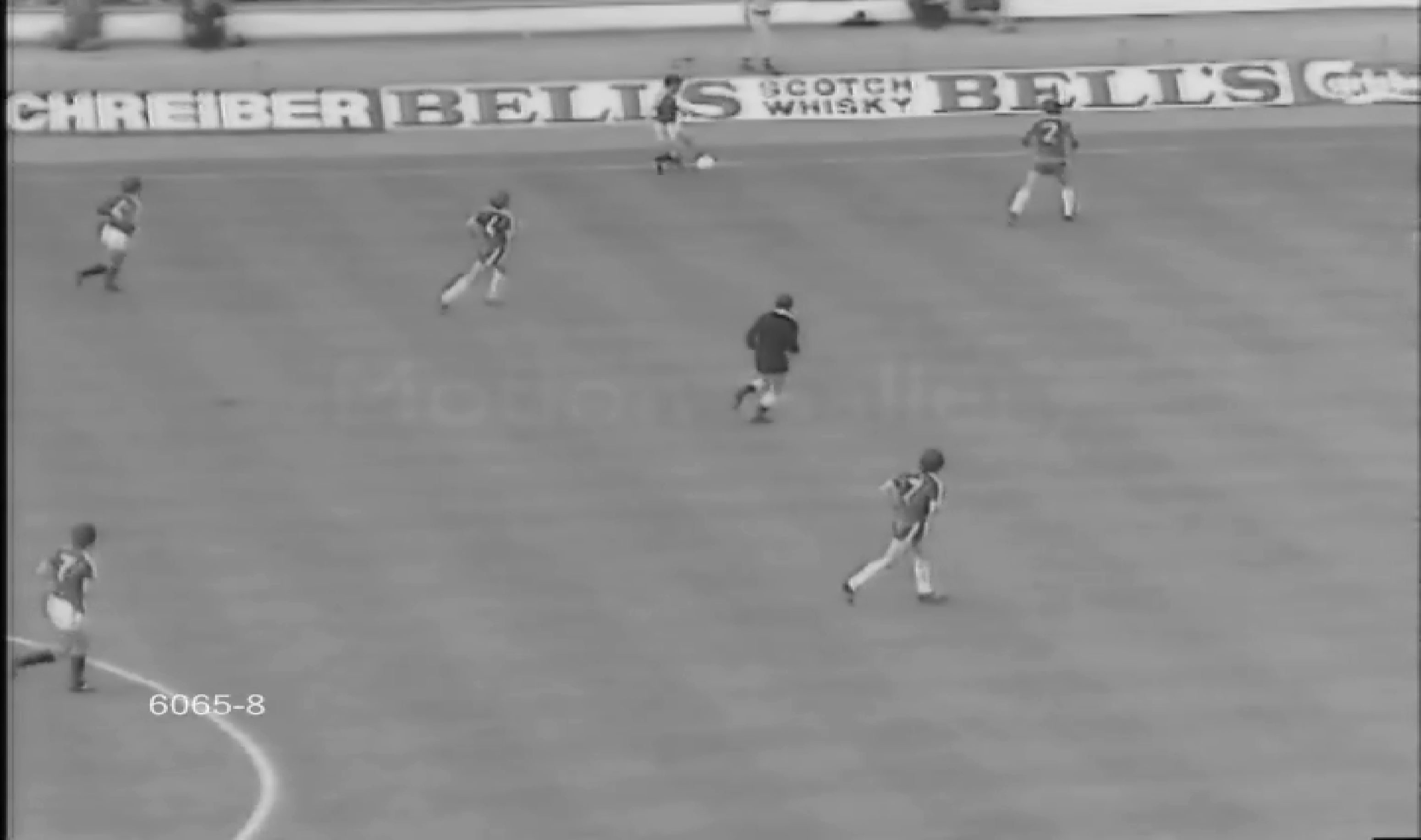}
		\includegraphics[width=0.49\textwidth]{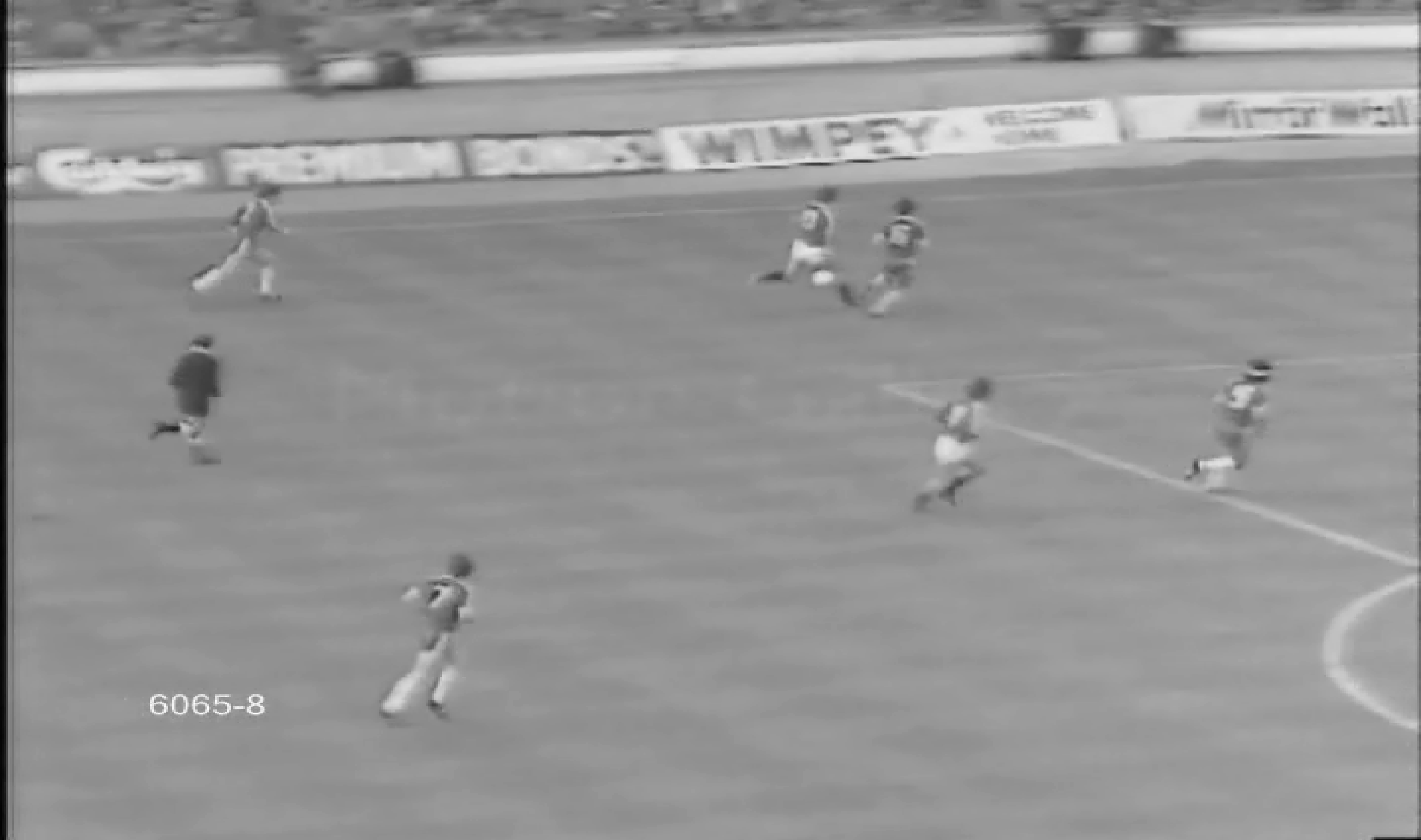}
		\caption{First and last frame of soccer video.}
		\label{fig:soccer}
	\end{center}
\end{figure}

In our test we iteratively reconstructed several objects from the video using the ORKA and the fg-ORKA algorithm. We choose a value of $C=15$ which is just large enough to allow $\lambda$ to keep track of the fast movements involved in the video. For the original ORKA algorithm we are only able to choose the approximation parameter $K=3$ before running out of memory. For fg-ORKA we set $K=9$ which is the largest parameter such that $(2C+1)^{2\cdot3}\geq3^{2\cdot K}$, i.e., the complexity of fg-ORKA with this choice is even lower as the original ORKA algorithm with $K=3$. We did not use upsampling in this experiment ($J=0$). This means, the results directly show the benefit of using a larger approximation parameter $K$. We set $\mu=500$ for all but the first and third iteration, in which we set $\mu=1.000.000$. We generally recommend using a large $\mu$ in the first iteration to filter out any global background or illumination effects. In the third iteration the algorithm switched from detecting large objects (background, advertisement board) to small objects (players, referee). Because of their small size, the players movement only changes a few pixels of the frame and is hard to distinguish from left-over noise of the larger objects. Hence, we choose a large $\mu$ to suppress most of the noise effects. Overall, ORKA and fg-ORKA both struggle if there is a big variance in size of the objects. The reconstruction could be improved by adding more restrictions on the object matrix $U$ itself, such as compact or connected support. However, this is beyond the scope of this work.

Figure \ref{fig:soccer_fgORKA} shows the first three objects recovered by the fg-ORKA algorithm: the general background, the advertisement boards, and the referee. The referee object shows artifacts of other players which move in approximately the same direction and speed as the referee itself. Hence, the algorithm is unable to completely distinguish these. We have also manually tracked the position of the advertisement board and the referee every ten frames and compared the tracked position with the actual position. (The advertisement board was tracked by its position of the "S" in the second "BELLS" which is visible in most of the video.) The positions tracked by ORKA and fg-ORKA compared to the actual position are shown in Figure \ref{fig:soccer_positions}. We note that the original ORKA algorithm performs better in tracking the advertisement board (Figure \ref{fig:obj2-x} and \ref{fig:obj2-y}). To track the correct position of the advertisement, the details such as the written text can be important. such information can get lost when downsampling the data. Furthermore, due to the periodic shift in our model, the algorithm actually expects the advertisement that leave the frame on the left, to appear on the right side again. This effect can even increase with a larger parameter $K$, as we are comparing more frames with one another. For this reason, the full ORKA algorithm with a smaller parameter $K$ is actually beneficial here. However, the referee is visible throughout the entire video. Moreover, his dark jersey is a feature easily recognized by the algorithm even on lower resolutions. Hence, fg-ORKA performs much better in tracking the position of the referee.

\begin{figure}
	\begin{center}
	\includegraphics[width=0.3\textwidth]{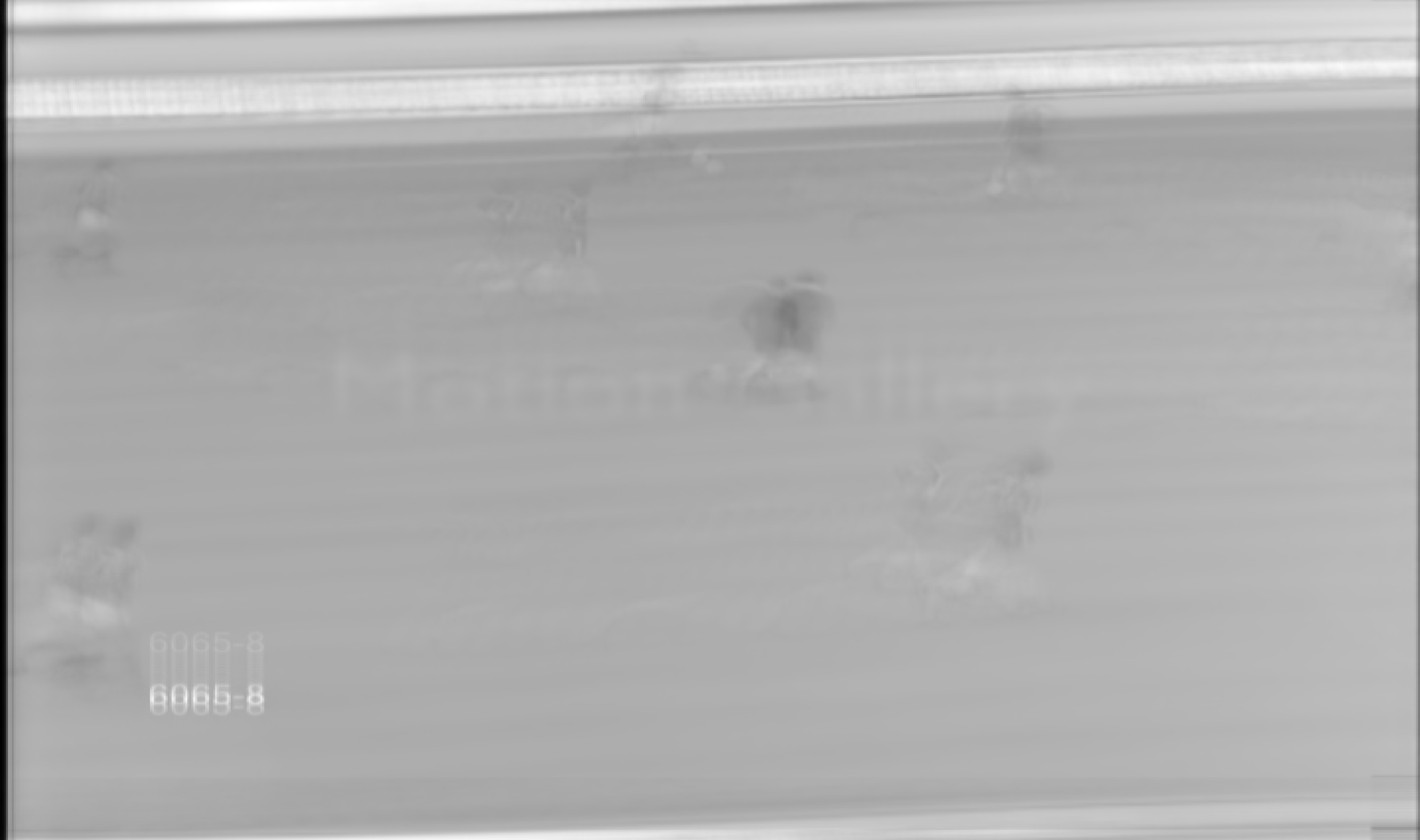}
	\includegraphics[width=0.3\textwidth]{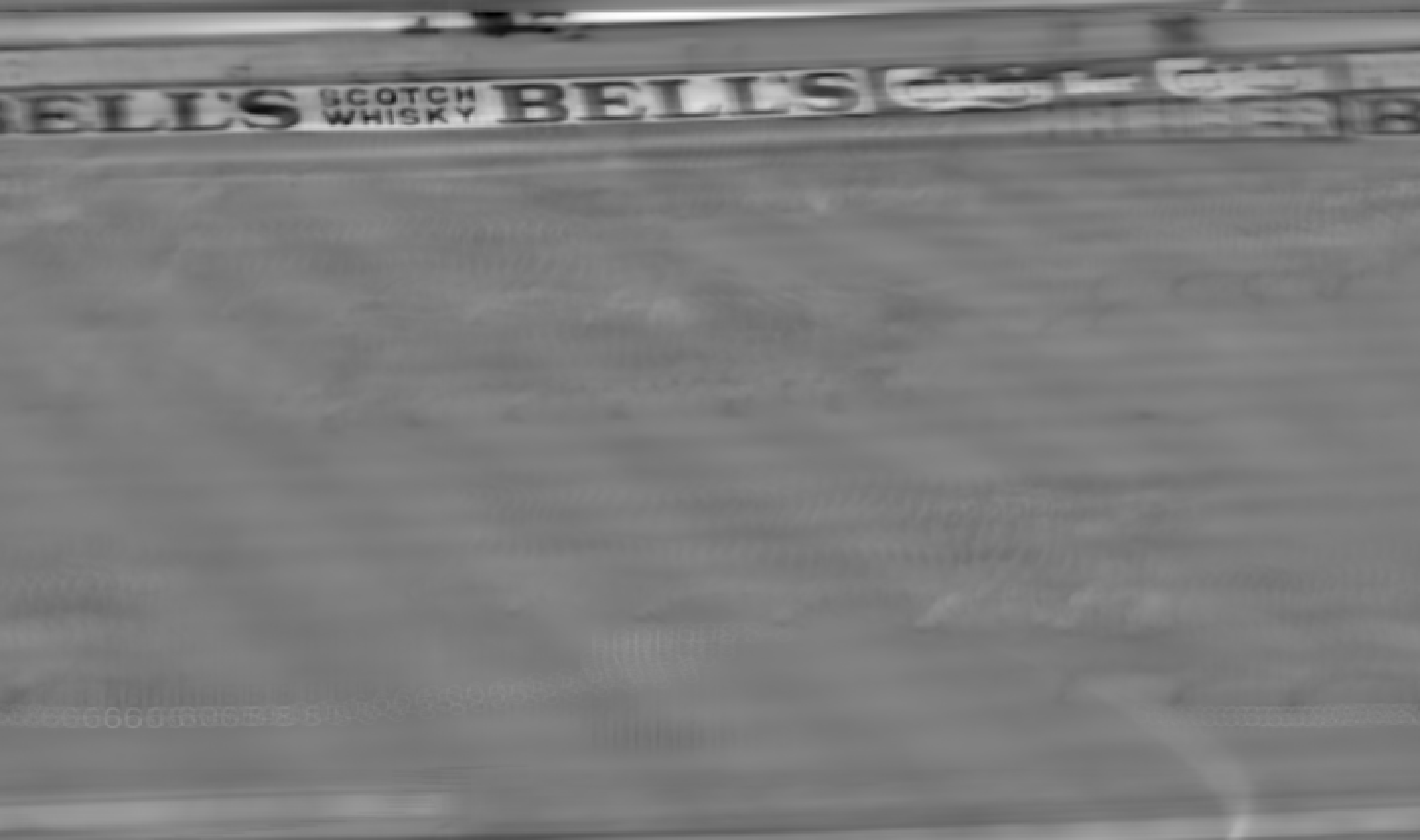}
	\includegraphics[width=0.3\textwidth]{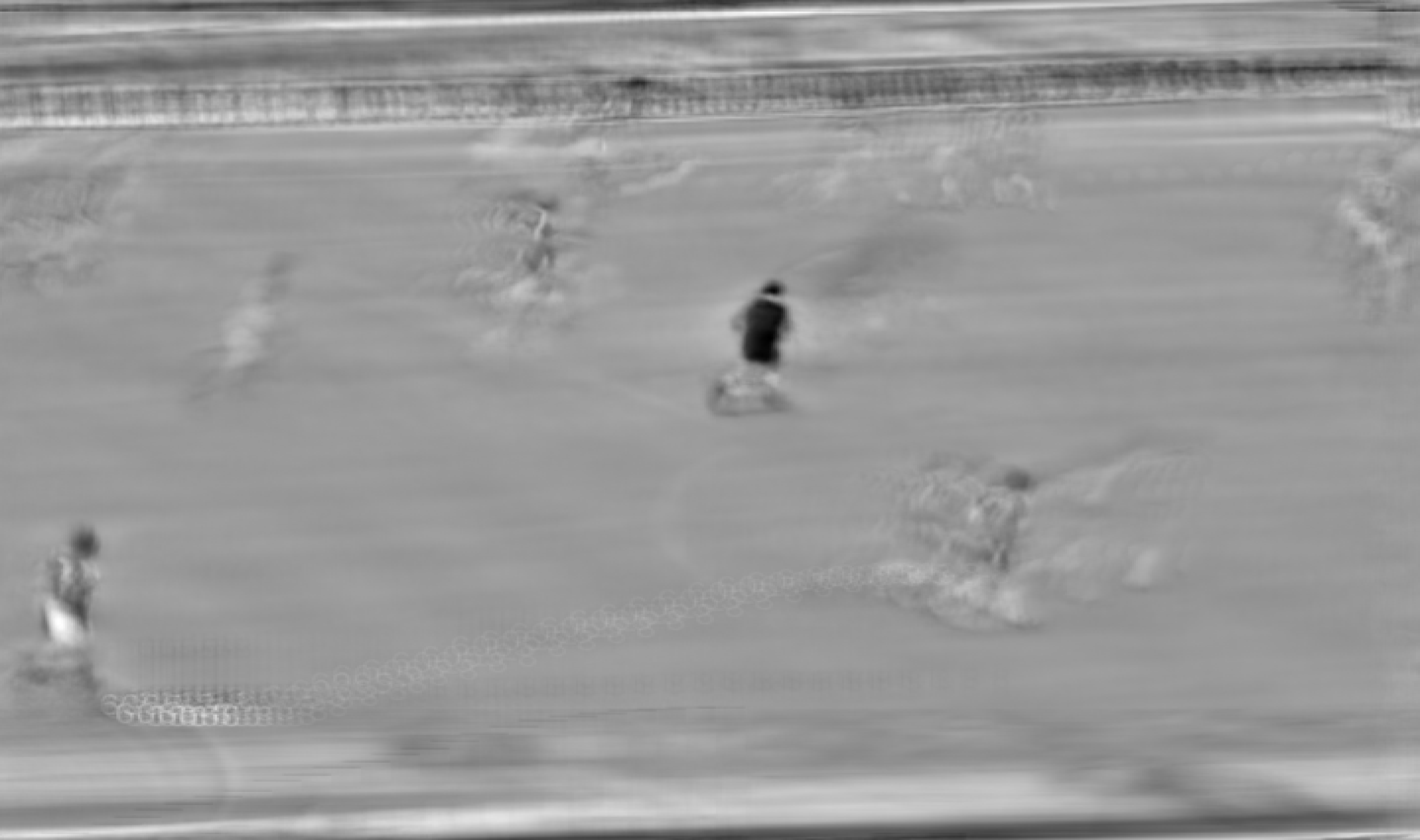}
	\caption{Reconstructed objects by the fg-ORKA algorithm: background, advertisement boards, referee.}
	\label{fig:soccer_fgORKA}
	\end{center}
\end{figure}

\begin{figure}
	\begin{center}
	\subfloat[\label{fig:obj2-x}]{\includegraphics[width=0.3\textwidth]{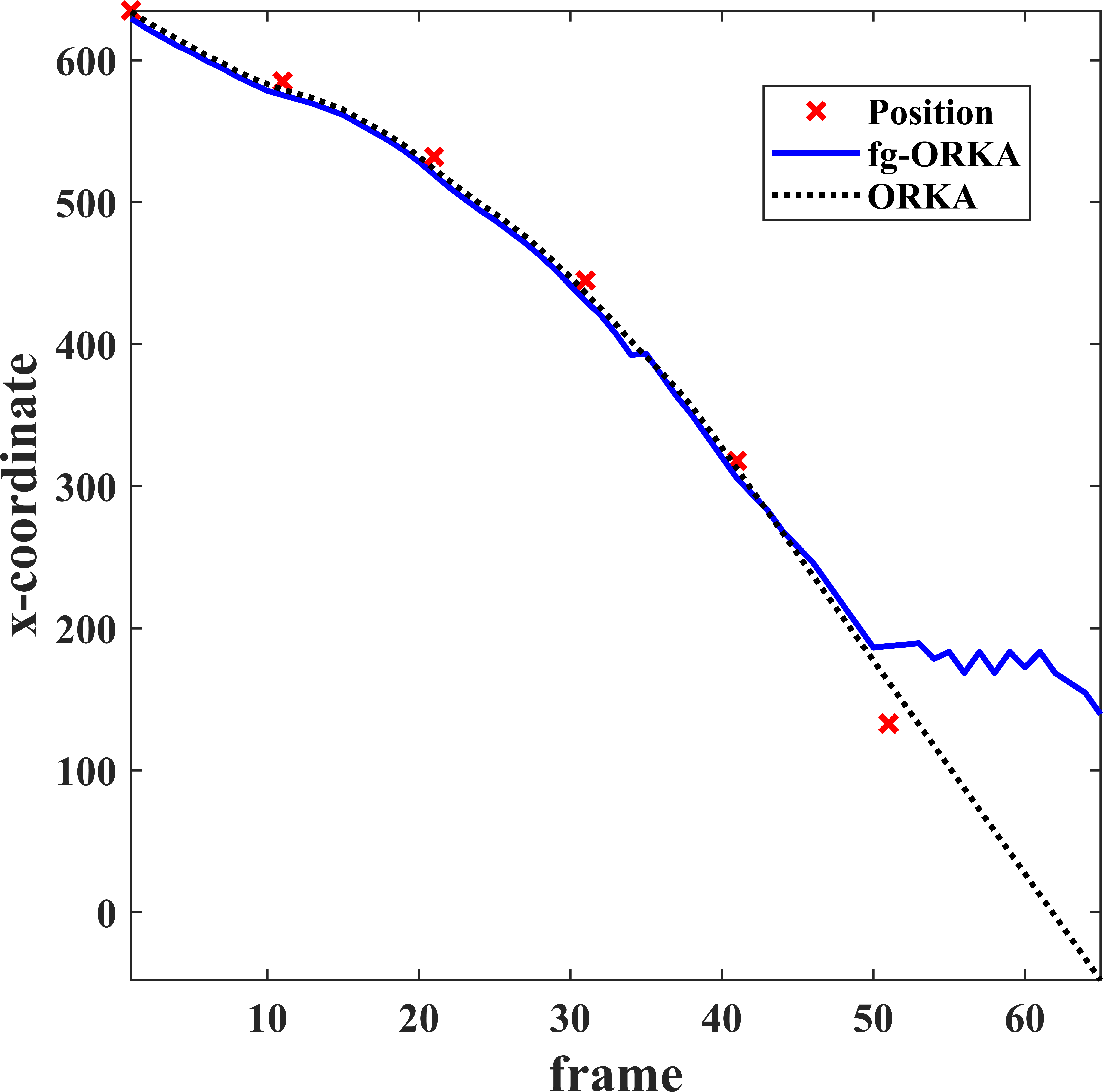}}
	\subfloat[\label{fig:obj2-y}]{\includegraphics[width=0.3\textwidth]{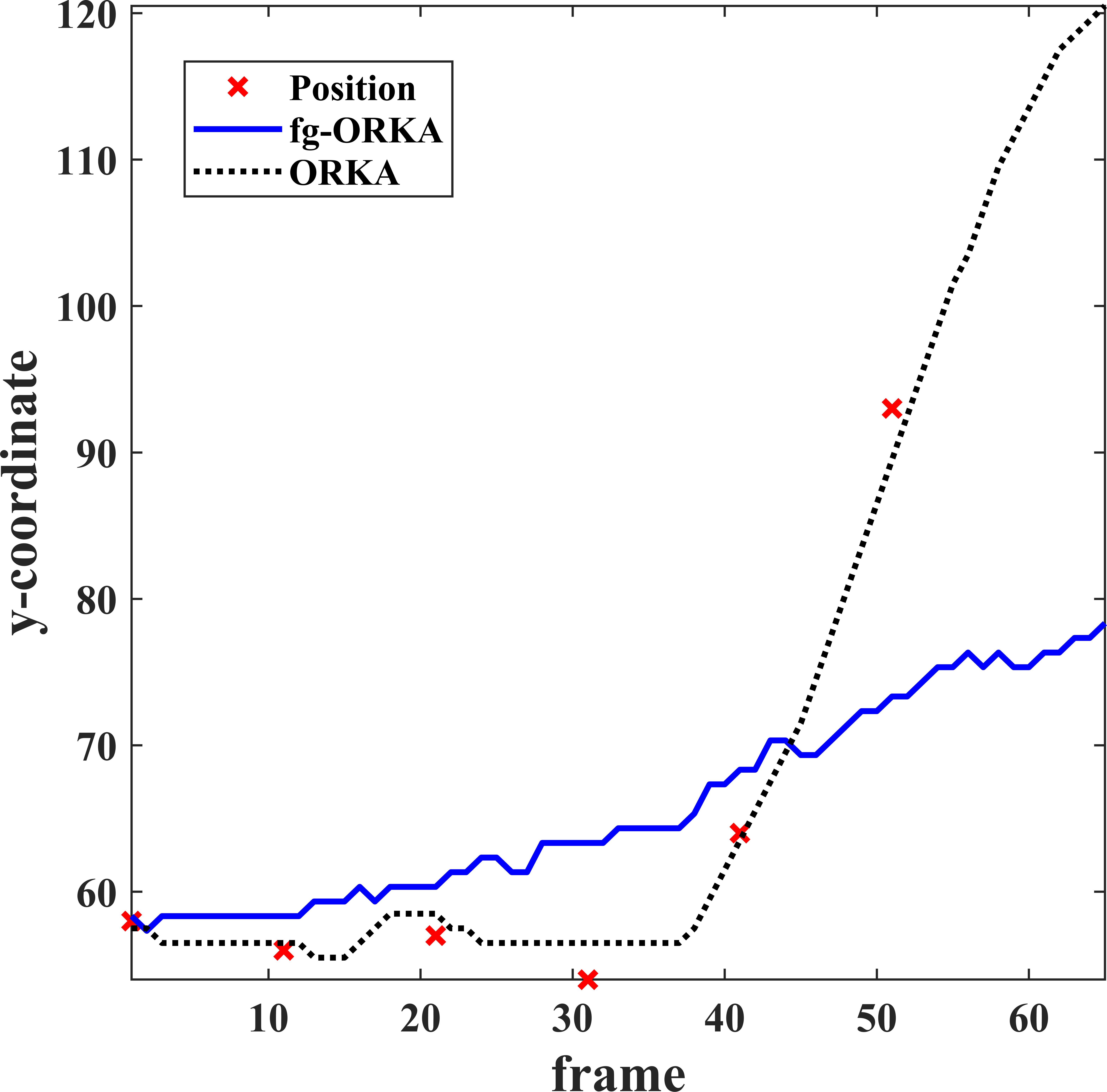}}\\
	\subfloat[\label{fig:obj3-x}]{\includegraphics[width=0.3\textwidth]{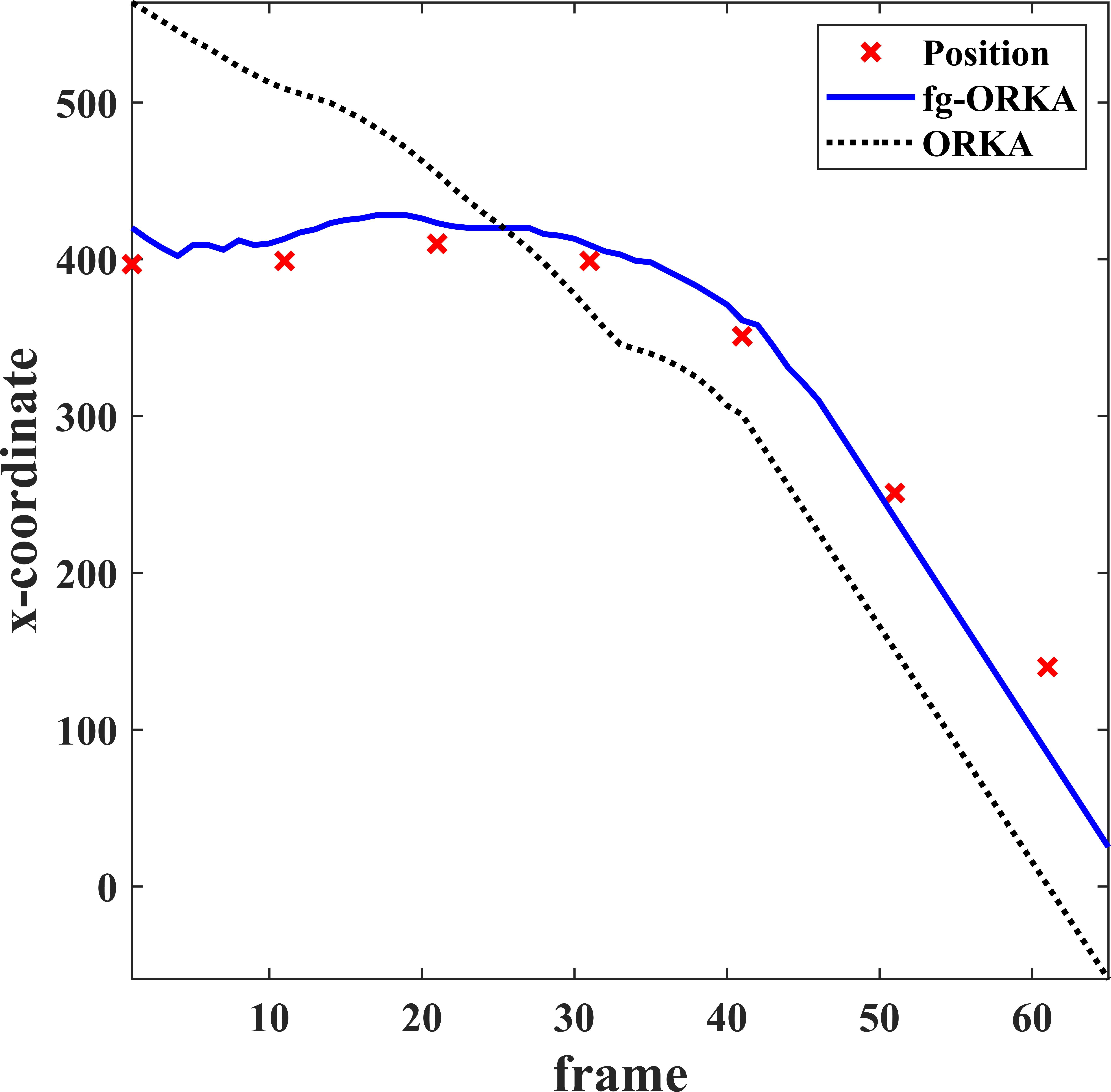}}
	\subfloat[\label{fig:obj3-y}]{\includegraphics[width=0.3\textwidth]{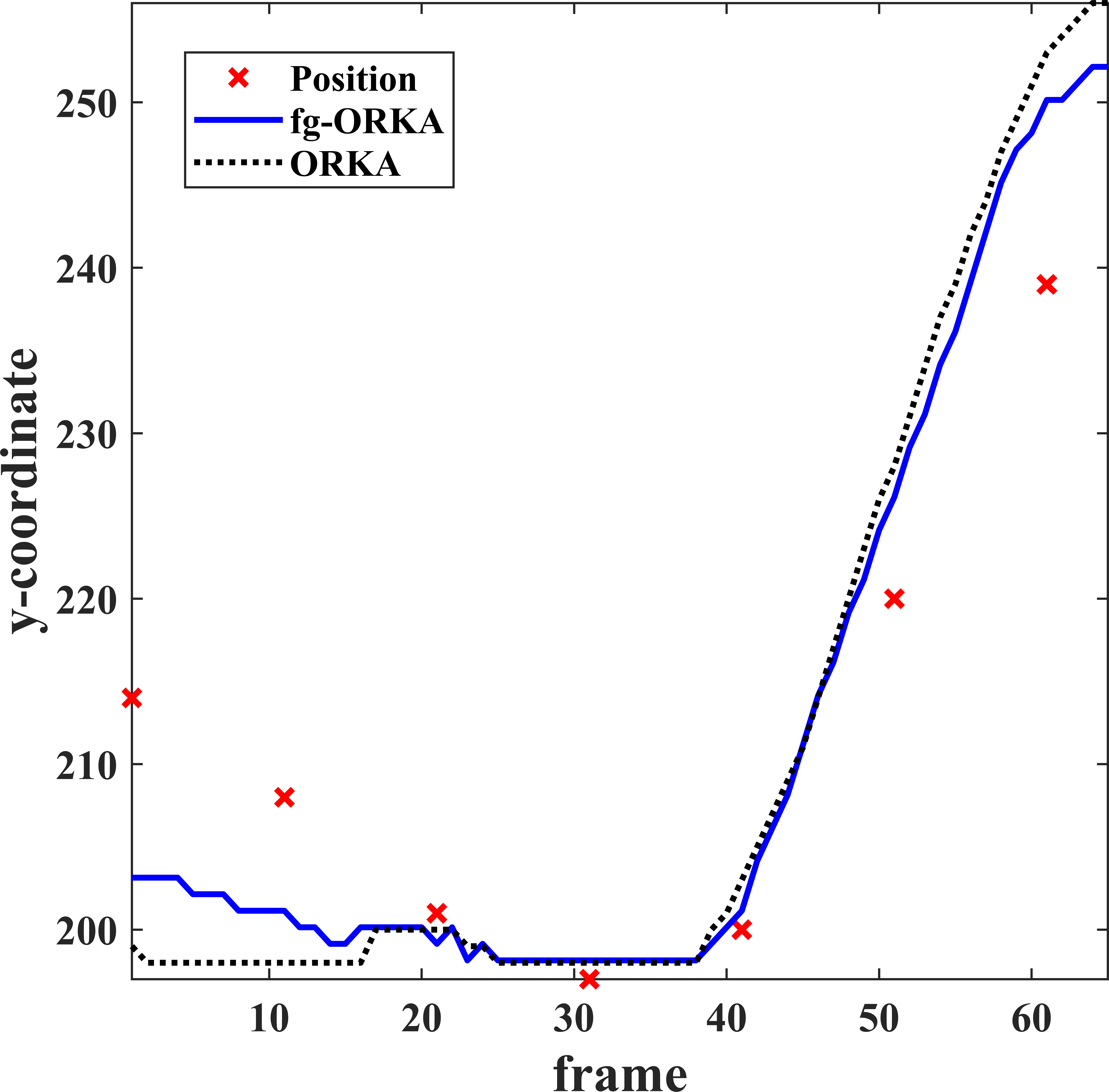}}
	\caption{Position of tracked object: Advertisement board (top) and referee (bottom).}
	\label{fig:soccer_positions}
	\end{center}
\end{figure}

\section{Conclusion}

We introduced a new iterative version of the ORKA algorithm. It significantly reduces the complexity and runtime compared to the original algorithm.  The method downsamples given data into a shift invariant subspace and uses the low resolution version to obtain a low resolution version of the object movement. This is then used as first approximation for the original resolution. Hence, we only need to calculate an update step whose complexity no longer scales with the parameter $C$. Depending on the estimated object speed, the downsampling step is applied iteratively. Furthermore, we are also able to track the movement much more accurate by artificially upsampling the data. We introduced three possible resampling strategies based on Wavelet transform, Fourier transform, and an error minimizing downsampling. The strategies can use different resampling rates where we identified a resampling rate of $r=2$ as the most stable and $r=3$ as the most efficient one.

A thorough complexity and error analysis of the new method was presented. The complexity of the new approach only scales $O(3^K)$ instead of $O((2C+1)^K)$ which makes it independent of the parameter $C$. The error analysis showed that the approximation error can be limited as long as the given data is sufficiently smooth. the theoretical results have been confirmed in different experimental setups. Furthermore, we have demonstrated the algorithm on two different applications with real data.

The experiment performed on the soccer video demonstrated the current limitations of the proposed technique. It has trouble dealing with fast moving objects or objects that significantly vary in size. Furthermore, it identifies objects solely depending on their movements which leads to combined reconstructions whenever two or more objects have approximately the same movement. The reconstruction can surely be improved by enforcing more restrictions on the object, such as compact or connected support. Using a non-periodic shift operator is another interesting adjustment for future applications.\\

\bibliographystyle{unsrt}
\bibliography{references}

\end{document}